\documentclass[11pt]{amsart}
\usepackage{amsfonts,amssymb,amscd,amsmath,enumerate,verbatim,calc,latexsym,pstcol,pst-plot,pst-3d}
\usepackage[noxcolor]{pstricks}
\input xy
\xyoption{all}

\def\frk{\mathfrak}               

\def\Phi{{\frk N}}
\def\opn#1#2{\def#1{\operatorname{#2}}} 
\opn\chara{char} \opn\length{\ell} \opn\pd{pd} \opn\rk{rk}
\opn\projdim{proj\,dim} \opn\injdim{inj\,dim} \opn\rank{rank}
\opn\depth{depth} \opn\grade{grade} \opn\height{height}
\opn\embdim{emb\,dim} \opn\codim{codim}

\opn\Tr{Tr} \opn\bigrank{big\,rank}
\opn\superheight{superheight}\opn\lcm{lcm}
\opn\trdeg{tr\,deg}
\opn\reg{reg} \opn\lreg{lreg} \opn\ini{in} \opn\lpd{lpd}
\opn\size{size}\opn{\mult}{mult}
\opn\link{lk}\opn\star{st}
\opn\div{div} \opn\Div{Div} \opn\cl{cl} \opn\Cl{Cl}
\opn\Spec{Spec} \opn\Supp{Supp} \opn\supp{supp} \opn\Sing{Sing}
\opn\Ass{Ass} \opn\Min{Min}
\opn\Ann{Ann} \opn\Rad{Rad} \opn\Soc{Soc}
\opn\Syz{Syz} \opn\Im{Im} \opn\Ker{Ker} \opn\Coker{Coker}
\opn\Am{Am} \opn\Hom{Hom} \opn\Tor{Tor} \opn\Ext{Ext}
\opn\End{End} \opn\Aut{Aut} \opn\id{id}

\opn\nat{nat}
\opn\pff{pf}
\opn\Pf{Pf} \opn\GL{GL} \opn\SL{SL} \opn\mod{mod} \opn\ord{ord}
\opn\Gin{Gin}
\opn\Hilb{Hilb}\opn\adeg{adeg}\opn\std{std}\opn\ip{infpt}
\opn\Pol{Pol}
\opn\sat{sat}
\opn\Var{Var}

\opn\aff{aff} \opn\con{conv} \opn\relint{relint} \opn\st{st}
\opn\lk{lk} \opn\cn{cn} \opn\core{core} \opn\vol{vol}
\opn\link{link} \opn\star{star} \opn\car{char}
\opn\gr{gr}

\def\pot#1#2{#1[\kern-0.28ex[#2]\kern-0.28ex]}

\opn\dirlim{\underrightarrow{\lim}}
\opn\inivlim{\underleftarrow{\lim}}
\def\Implies{\ifmmode\Longrightarrow \else
        \unskip${}\Longrightarrow{}$\ignorespaces\fi}
\def\implies{\ifmmode\Rightarrow \else
        \unskip${}\Rightarrow{}$\ignorespaces\fi}
\def\iff{\ifmmode\Longleftrightarrow \else
        \unskip${}\Longleftrightarrow{}$\ignorespaces\fi}

\let\:=\colon
\newtheorem{theorem}{Theorem}[section]
\newtheorem{theorem A}{Theorem}

\newtheorem{cor}[theorem]{Corollary}
\newtheorem{pro}[theorem]{Proposition}

\newtheorem{example}[theorem]{Example}

\let\epsilon\varepsilon
\let\phi=\varphi
\let\kappa=\varkappa
\def\qed{\ifhmode\textqed\fi
      \ifmmode\ifinner\quad\qedsymbol\else\dispqed\fi\fi}
\def\textqed{\unskip\nobreak\penalty50
       \hskip2em\hbox{}\nobreak\hfil\qedsymbol
       \parfillskip=0pt \finalhyphendemerits=0}
\def\dispqed{\rlap{\qquad\qedsymbol}}

\opn\dis{dis}
\def\pnt{{\raise0.5mm\hbox{\large\bf.}}}

\opn\Lex{Lex}

\begin{document}

\title{Mostar index of graph operations}

\author{Shehnaz Akhter, Zahid Iqbal, Adnan Aslam and Wei Gao}

\address{ School of Natural Sciences, National University of Sciences and Technology, Sector H-12, Islamabad, Pakistan.
}
\email{shehnazsakhter36@yahoo.com;786zahidwarriach@gmail.com}

\address{Department of Natural Sciences and Humanities, University of Engineering and
Technology, Lahore, Pakistan (RCET), Pakistan} \email{adnanaslam15@yahoo.com}

\address{School of Information Science and Technology, Yunnan Normal University, Kunming, 650500, China.} \email{gaowei@ynnu.edu.cn}

\begin{abstract} Very recently, a bond-additive topological descriptor, known as the Mostar index, has been proposed as a measure of peripherality in graphs and networks. In this article, we compute the Mostar index of corona product, Cartesian product, join, lexicographic product, Indu-Bala product and subdivision vertex-edge join of graphs and apply these results to find the Mostar index of various classes of chemical graphs and nanostructures.
\end{abstract}


\maketitle

\section*{Introduction}

Research in mathematical chemistry provides a specific consideration to describe the distinctive nature of chemical structure and hence, one sometimes wishes to relate a unique quantitative value to every chemical compound.
In Mathematical Chemistry, one of the important problems is to analyze the distinctive nature of chemical structure with the help of structural invariants called topological descriptors.
The benefit of topological descriptors is that they may be applied directly as simple numerical descriptors for the correlation of chemical structure with various physical properties, biological activity or chemical reactivity in Quantitative Structure Activity Relationships (QSAR) and in Quantitative Structure-Property Relationships (QSPR)\cite {1,2}.
There are many graphs associated with numerical descriptors, which play a pivotal role in nanotechnology and theoretical chemistry.
Thereby, the assessment of these numerical descriptors is one of the famous lines of research.
The bond-additive topological descriptors are extensively used to describe the features of chemical graphs and their fragments, setting up the links between the structure and properties of molecules.
The first topological descriptor as a bond-additive index, known as Wiener index \cite{3} in which every bond yields a contribution that is equal to the product of the number of atoms on each side of the bond.
Inspired by the various successful topological descriptors of this type such as irregularity\cite{4}, Zagreb \cite{5}, PI \cite{6}, Szeged \cite{7}, revised-Szeged \cite{8,9,10}, and recently, another bond-additive topological descriptor, the Mostar index has been proposed by Do\v{s}lic et. al in \cite{11}.
This index gives information about the peripherality of individual bonds and then sums the inputs of all bonds into a global measure of peripherality for the given chemical graph.

Throughout this article, each graph will be a finite, undirected and simple.
Let $\mathcal{G}_l=(V(\mathcal{G}_l),E(\mathcal{G}_l))$ be a graph with the edge set $E(\mathcal{G}_l)$ and the vertex set $V(\mathcal{G}_l)$.
The cardinalities of vertex and edge sets of $\mathcal{G}_l$ are said to be the order and size of it respectively.
 A molecular graph is a graph whose vertices corresponds to atoms, and an edge between two vertices is related to the chemical bond between them.
The degree of a vertex $\mathfrak{u}_l$ of a graph $\mathcal{G}_l$ is represented by $d_{\mathcal{G}_l}(\mathfrak{u}_l)$, and it speaks the number of edges incident with $\mathfrak{u}_l.$

An edge of the graph is assumed to be peripheral if there are more vertices closer to one of its end-vertices as compare to the other one.
In other words, for an edge $\mathfrak{u}_l\mathfrak{u'}_l,$ the large value of the absolute difference of the number of vertices closer to $\mathfrak{u}_l$ than to $\mathfrak{u'}_l$ (denoted by $\mathrm{n}_{\mathfrak{u}_l}(\mathfrak{e}|\mathcal{G}_l)$) and the number of vertices closer to $\mathfrak{u'}_l$ than to $\mathfrak{u}_l$ (which we denote by $\mathrm{n}_{\mathfrak{u'}_l}(\mathfrak{e}|\mathcal{G}_l)$) expresses a peripheral position of $uv$ in $\mathcal{G}_l.$
The absolute difference $\left|\mathrm{n}_{\mathfrak{u}_l}(\mathfrak{e}|\mathcal{G}_l)-\mathrm{n}_{\mathfrak{u'}_l}
(\mathfrak{e}|\mathcal{G}_l)\right|$ said to be the contribution of the edge $\mathfrak{u}_l\mathfrak{u'}_l.$
The Mostar index of a graph $\mathcal{G}_l$ is described as the sum of such contributions over all edges of $\mathcal{G}_l$ as follows;

\begin{equation}\label{mostar}
Mo(\mathcal{G}_l)=\sum\limits_{\mathfrak{u}_l\mathfrak{u'}_l\in \mathsf{E}(\mathcal{G}_l)}\left|\mathrm{n}_{\mathfrak{u}_l}(\mathfrak{e}|\mathcal{G}_l)
-\mathrm{n}_{\mathfrak{u'}_l}(\mathfrak{e}|\mathcal{G}_l)\right|.
\end{equation}

Do\v{s}lic et. al compute the Mostar index of benzenoid systems by using a simple cut method in the same article. Furthermore, they also find the extremal values for trees and unicyclic graphs.
Later, the results of the Mostar index of bicyclic graphs were given in \cite{12}.
Tratnik showed that the Mostar index of a weighted graph can be determined in terms of Mostar indices of quotient graphs in \cite{13}.
Arockiaraj et. al presented the precise values of the Mostar index for the family of carbon nanocone and coronoid structures in reference \cite{14}.

The ~term irregularity of a graph $\mathcal{G}_l$ was first presented by Albertson~\cite{15}.
It is symbolized by $irr(\mathcal{G}_l)$ and described as follows:
\begin{equation}
irr(\mathcal{G}_l)=\sum\limits_{\mathfrak{u}_l\mathfrak{u'}_l\in \mathsf{E}(\mathcal{G}_l)}\left|\deg_{\mathcal{G}_l}({\mathfrak{u}_l})
-\deg_{\mathcal{G}_l}({\mathfrak{u'}_l})\right|.
\end{equation}

Abdo~et~al.~\cite{16} described the total irregularity measure of a graph $\mathcal{G}_l$, which was expressed as follows:
\begin{equation}
irr_t(\mathcal{G}_l)=\dfrac{1}{2}\sum\limits_{\mathfrak{u}_l,\mathfrak{u'}_l\in \mathsf{V}(\mathcal{G}_l)}\left|\deg_{\mathcal{G}_l}({\mathfrak{u}_l})
-\deg_{\mathcal{G}_l}({\mathfrak{u'}_l})\right|.
\end{equation}
For the detail discussions about the different graph invariants, we refer \cite{17,18,19,20,21,22,23,24,25,26,27,28,29}.
Now, we discuss some known results, that are heavily used in this paper. For the vertex-transitive graph $\mathcal{G}_l$, its Mostar index will be zero \cite{30}. Using this result, Do\v{s}li\'{c} et al. deduced the following result for the complete graph $\mathcal{K}_s$ of $s$ vertices, cyclic graph $\mathcal{C}_s$ of order $s$, path graph $\mathcal{P}_s$ on $s$ vertices, and for complete bipartite graph $\mathcal{K}_{r,s}$ with parts of sizes $r$ and $s.$

\begin{pro}\label{pro}\emph{\cite{30}}
 $Mo(\mathcal{K}_s)=Mo(\mathcal{C}_s)=Mo(\mathcal{K}_{s,s})=0$ and also $Mo(\mathcal{P}_s)=\left\lfloor\dfrac{(s-1)^2}{2}\right\rfloor$.
\end{pro}
For a simple undirected graph $\mathcal{G}_l$ with $s$ vertices, Abdo et al. gave the following interesting bound for the total irregularity index.
\begin{pro}\label{pro1}\emph{\cite{16}}
\begin{eqnarray*}
\begin{split}
irr_t(\mathcal{G}_l)&\leq\left\{
\begin{array}{ll}
\dfrac{1}{12}(2s^3-3s^2-2s),    & \mbox{if $s$ is even,} \\\\
\dfrac{1}{12}(2s^3-3s^2-2s+3),  & \mbox{if $s$ is odd,}
\end{array}
\right.
\end{split}
\end{eqnarray*}
\end{pro}

\section{Main result}
In this section, we derive the expressions for the Mostar index of different graph operations. First, we compute the Mostar index of corona product of two graphs.
\subsection{Corona Product}
Let $\mathcal{G}_l$ and $\mathcal{G}_m$ be two graphs with order $s_1$ and $s_2$, and size $t_1$ and $t_2$, respectively.
The corona product $\mathcal{G}_l\circ \mathcal{G}_m$ of graphs $\mathcal{G}_l$ and $\mathcal{G}_m$ is a graph, which can be drawn by using a copy of $\mathcal{G}_l$ and $s_1$ copies of $\mathcal{G}_m$ and linking the $g$-th vertex of $\mathcal{G}_l$ to every vertex in $g$-th copy of $\mathcal{G}_m$, $1\leq g\leq s_1$.
In the following theorem, we give the expression of the Mostar index of corona product of two graphs.
Here, the number of triangles which consist of an edge $\mathfrak{e}=\mathfrak{u}_l\mathfrak{u'}_l$ in $\mathcal{G}_l$ is denoted by $t_{\mathfrak{u}_l\mathfrak{u'}_l}$.
\begin{theorem}\label{theoremcor}
Let $\mathcal{G}_l$ and $\mathcal{G}_m$ be the two graphs. Then
\begin{eqnarray*}
\begin{split}
Mo(\mathcal{G}_l\circ \mathcal{G}_m)&\leq s_1 irr(\mathcal{G}_t)+(s_2+1)Mo(\mathcal{G}_l)+s_1s_2\left|2-s_1-s_1s_2\right|+2s_1t_2.
\end{split}
\end{eqnarray*}
\end{theorem}
\begin{proof}
Using the definition of corona product in equation \eqref{mostar}
\begin{eqnarray}\label{cor1}
\begin{split}
Mo(\mathcal{G}_l\circ \mathcal{G}_m)&=s_1\sum\limits_{\mathfrak{e}=\mathfrak{u}_m\mathfrak{u'}_m\in \mathsf{E}(\mathcal{G}_m)}\left|\mathrm{n}_{\mathfrak{u}_m}(\mathfrak{e}|\mathcal{G}_l\circ \mathcal{G}_m)
-\mathrm{n}_{\mathfrak{u'}_m}(\mathfrak{e}|\mathcal{G}_l\circ \mathcal{G}_m)\right|
\\
&+\sum\limits_{\mathfrak{e}=\mathfrak{u}_l\mathfrak{u'}_l\in \mathsf{E}(\mathcal{G}_l)}\left|\mathrm{n}_{\mathfrak{u}_l}(\mathfrak{e}|\mathcal{G}_l\circ \mathcal{G}_m)-\mathrm{n}_{\mathfrak{u'}_l}(\mathfrak{e}|\mathcal{G}_l\circ \mathcal{G}_m)\right|
\\
&+\sum\limits_{\mathfrak{u}_l\in \mathsf{V}(\mathcal{G}_l)}\sum\limits_{\mathfrak{u}_m\in \mathsf{V}(\mathcal{G}_m)}\left|\mathrm{n}_{\mathfrak{u}_l}(\mathfrak{e}|\mathcal{G}_l\circ \mathcal{G}_m)-\mathrm{n}_{\mathfrak{u}_m}(\mathfrak{e}|\mathcal{G}_l\circ \mathcal{G}_m)\right|
\end{split}
\end{eqnarray}

For every $\mathfrak{e}=\mathfrak{u}_m\mathfrak{u'}_m\in \mathsf{E}(\mathcal{G}_m)$ if there exists $\mathfrak{u''}_m\in \mathsf{V}(\mathcal{G}_m)$ such that $\mathfrak{u}_m\mathfrak{u''}_m\notin \mathsf{E}(\mathcal{G}_m)$ and $\mathfrak{u'}_m\mathfrak{u''}_m\notin \mathsf{E}(\mathcal{G}_m)$ then $d_{\mathcal{G}_l\circ \mathcal{G}_m}(\mathfrak{u}_m,\mathfrak{u''}_m)=d_{\mathcal{G}_l\circ \mathcal{G}_m}(\mathfrak{u}_m,\mathfrak{u''}_m)=2$ and if $\mathfrak{u}_m\mathfrak{u''}_m\in \mathsf{E}(\mathcal{G}_m)$ and $\mathfrak{u'}_m\mathfrak{u''}_m\in \mathsf{E}(\mathcal{G}_m)$ then $d_{\mathcal{G}_l\circ \mathcal{G}_m}(\mathfrak{u}_m,\mathfrak{u''}_m)=d_{\mathcal{G}_l\circ \mathcal{G}_m}(\mathfrak{u'}_m,\mathfrak{u''}_m)=1$.
Hence $\mathrm{n}_{\mathfrak{u}_m}(\mathfrak{e}|\mathcal{G}_l\circ \mathcal{G}_m)=\deg_{\mathcal{G}_m}(\mathfrak{u}_m)-t_{\mathfrak{u}_m\mathfrak{u'}_m}$ and
\begin{eqnarray}\label{cor2}
\begin{split}
&\sum\limits_{\mathfrak{e}=\mathfrak{u}_m\mathfrak{u'}_m\in \mathsf{E}(\mathcal{G}_m)}\left|\mathrm{n}_{\mathfrak{u}_m}(\mathfrak{e}|\mathcal{G}_l\circ \mathcal{G}_m)-\mathrm{n}_{\mathfrak{u'}_m}(\mathfrak{e}|\mathcal{G}_l\circ \mathcal{G}_m)\right|
\\
=&\sum\limits_{\mathfrak{e}=\mathfrak{u}_m\mathfrak{u'}_m\in \mathsf{E}(\mathcal{G}_m)}\left|\deg_{\mathcal{G}_m}(\mathfrak{u}_m)-t_{\mathfrak{u}_m\mathfrak{u'}_m}
-\deg_{\mathcal{G}_m}(\mathfrak{u'}_m)+t_{\mathfrak{u}_m\mathfrak{u'}_m}\right|
\\
=&\sum\limits_{\mathfrak{e}=\mathfrak{u}_m\mathfrak{u'}_m\in \mathsf{E}(\mathcal{G}_m)}\left|\deg_{\mathcal{G}_m}(\mathfrak{u}_m)-\deg_{\mathcal{G}_m}(\mathfrak{u'}_m)\right|
\\
=&irr(\mathcal{G}_m).
\end{split}
\end{eqnarray}

We now assume that $\mathfrak{e}=\mathfrak{u}_l\mathfrak{u'}_l\in \mathsf{E}(\mathcal{G}_l)$.
Then for each vertex $\mathfrak{u''}_l$ closer to $\mathfrak{u}_l$ than $\mathfrak{u'}_l$, the vertices of the copy of $\mathcal{G}_m$ attached to $\mathfrak{u''}_l$ are also closer to $\mathfrak{u}_l$ than $\mathfrak{u'}_l$.
Since each copy of $\mathcal{G}_m$ has exactly $s_2$ vertices, $\mathrm{n}_{\mathfrak{u}_l}(\mathfrak{e}|\mathcal{G}_l\circ \mathcal{G}_m)=(s_2+1)\mathrm{n}_{\mathfrak{u}_l}(\mathfrak{e}|\mathcal{G}_l)$.
Similarly, $\mathrm{n}_{\mathfrak{u'}_l}(\mathfrak{e}|\mathcal{G}_l\circ \mathcal{G}_m)=(s_2+1)\mathrm{n}_{\mathfrak{u'}_l}(\mathfrak{e}|\mathcal{G}_l)$.
Therefore, we have
\begin{eqnarray}\label{cor3}
\begin{split}
&\sum\limits_{\mathfrak{e}=\mathfrak{u}_l\mathfrak{u'}_l\in \mathsf{E}(\mathcal{G}_l)}\left|\mathrm{n}_{\mathfrak{u}_l}(\mathfrak{e}|\mathcal{G}_l\circ \mathcal{G}_l)-\mathrm{n}_{\mathfrak{u}_l}(\mathfrak{e}|\mathcal{G}_l\circ \mathcal{G}_l)\right|
\\
=&\sum\limits_{\mathfrak{e}=\mathfrak{u}_l\mathfrak{u}_l\in \mathsf{E}(\mathcal{G}_l)}\left|(s_2+1)\mathrm{n}_{\mathfrak{u}_l}(\mathfrak{e}|\mathcal{G}_l)
-(s_2+1)\mathrm{n}_{\mathfrak{u}_l}(\mathfrak{e}|\mathcal{G}_l)\right|
\\
=&(s_2+1)\sum\limits_{\mathfrak{e}=\mathfrak{u}_l\mathfrak{u}_l\in \mathsf{E}(\mathcal{G}_l)}\left|\mathrm{n}_{\mathfrak{u}_l}(\mathfrak{e}|\mathcal{G}_l)
-\mathrm{n}_{\mathfrak{u}_l}(\mathfrak{e}|\mathcal{G}_l)\right|
\\
=&(s_2+1)Mo(\mathcal{G}_l).
\end{split}
\end{eqnarray}

Finally, we assume that $\mathfrak{e}=\mathfrak{u}_l\mathfrak{u}_m$ with $\mathfrak{u}_m\in \mathsf{V}(\mathcal{G}_m)$ and $\mathfrak{u}_l\in \mathsf{V}(\mathcal{G}_l)$, $\deg_{\mathcal{G}_m}(\mathfrak{u}_m)=k$ and $\{\mathfrak{u}_{m_1},\dots,\mathfrak{u}_{m_k}\}$ are adjacent vertices of $\mathfrak{u}_m\in \mathcal{G}_{m_i}$.
By definition of corona product of graphs, $\mathfrak{u}_l$ is adjacent to vertices $\mathfrak{u}_{m_1},\dots,\mathfrak{u}_{m_k}$.
Thus for each $j$, $1\leq j\leq k$, $\mathfrak{u}_{m_j}$ is equidistant from $\mathfrak{u}_l$ and $\mathfrak{u}_m$.
On the other hand, every vertex of $\mathcal{G}_l\circ \mathcal{G}_m$ other than $\mathfrak{u}_m, \mathfrak{u}_{m_1},\dots,\mathfrak{u}_{m_k}$ are closer to $\mathfrak{u}_l$ than $\mathfrak{u}_m$.
This implies that $\mathrm{n}_{\mathfrak{u}_m}(\mathfrak{e}|\mathcal{G}_l\circ \mathcal{G}_m)=|\mathsf{V}(\mathcal{G}_l\circ \mathcal{G}_m)|-(\deg_{\mathcal{G}_m}(\mathfrak{u}_m)+1)$ and $\mathrm{n}_{\mathfrak{u}_l}(\mathfrak{e}|\mathcal{G}_l\circ \mathcal{G}_m)=1$.
Therefore we have
\begin{eqnarray}\label{cor4}
\begin{split}
&\sum\limits_{\mathfrak{u}_l\in \mathsf{V}(\mathcal{G}_l)}\sum\limits_{\mathfrak{u}_m\in \mathsf{V}(\mathcal{G}_m)}\left|\mathrm{n}_{\mathfrak{u}_l}(\mathfrak{e}|\mathcal{G}_l\circ \mathcal{G}_m)-\mathrm{n}_{\mathfrak{u}_m}(\mathfrak{e}|\mathcal{G}_l\circ \mathcal{G}_m)\right|
\\
=&\sum\limits_{\mathfrak{u}_l\in \mathsf{V}(\mathcal{G}_l)}\sum\limits_{\mathfrak{u}_m\in \mathsf{V}(\mathcal{G}_m)}\left|1-|\mathsf{V}(\mathcal{G}_l\circ \mathcal{G}_m)|+(\deg_{\mathcal{G}_m}(\mathfrak{u}_m)+1)\right|
\\
\leq&\sum\limits_{\mathfrak{u}_m\in \mathsf{V}(\mathcal{G}_l)}\sum\limits_{\mathfrak{u}_l\in \mathsf{V}(\mathcal{G}_m)}\left|2-|\mathsf{V}(\mathcal{G}_l\circ \mathcal{G}_m)|\right|
+\sum\limits_{\mathfrak{u}_m\in \mathsf{V}(\mathcal{G}_l)}\sum\limits_{\mathfrak{u}_l\in \mathsf{V}(\mathcal{G}_m)}\deg_{\mathcal{G}_m}(\mathfrak{u}_m)
\\
=&s_1s_2\left|2-s_1-s_1s_2\right|+2s_2t_1.
\end{split}
\end{eqnarray}
By using results \eqref{cor2}-\eqref{cor4} in \eqref{cor1}, we acquire
\begin{eqnarray*}
\begin{split}
Mo(\mathcal{G}_l\circ \mathcal{G}_m)&\leq s_1 irr(\mathcal{G}_m)+(s_2+1)Mo(\mathcal{G}_l)+s_1s_2\left|2-s_1-s_1s_2\right|+2s_1t_12.
\end{split}
\end{eqnarray*}
This completes the proof.
\end{proof}

For $\mathcal{G}_l$, the $g$-thorny graph is obtained by joining $g$-number of pendant vertices to each vertex of $\mathcal{G}_l$ and it is recognized by $\mathcal{G}_{l}^g$.
The $g$-thorny graph of $\mathcal{G}_l$ is represented as $\mathcal{G}_l\circ \overline{\mathcal{K}}_m$.
Thus from Theorem \ref{theoremcor}, the Corollary \ref{cor1} follows.
\begin{cor}\label{cor1}
If $\mathcal{G}_l$ is a graph with $|\mathsf{E}(\mathcal{G}_l)|=t$ and $|\mathsf{V}(\mathcal{G}_l)|=s$.
Then
\begin{equation*}
Mo(\mathcal{G}_l\circ \overline{\mathcal{K}}_m)\leq(m+1)Mo(\mathcal{G}_l)+sm|2-s-sm|.
\end{equation*}
\end{cor}

\begin{example}
The bottleneck graph $\mathsf{B}$ of $\mathcal{G}_l$ is obtained by taking the corona product of $\mathcal{K}_2$ with $\mathcal{G}_l$.
Then $Mo(\mathcal{K}_2\circ \mathcal{G}_l)=2irr(\mathcal{G}_l)+4(s+t)$, where $|\mathsf{V}(\mathcal{G}_l)|=s$ and $|\mathsf{E}(\mathcal{G}_l)|=t$.
\end{example}
\begin{example}
For the vertices $\mathfrak{a}_l$, $1\leq l\leq s$, the structure of a bridge graph can be acquired by linking the vertices $\mathfrak{a}_l$ and $\mathfrak{a}_{l+1}$ of $\mathcal{A}_{l+1}$ by a connection for all $l=1,2,\dots,s-1$ and it is denoted by $B(\mathcal{A}_1,\mathcal{A}_2,\dots,\mathcal{A}_s;\mathfrak{a}_1,\mathfrak{a}_2,\dots,\mathfrak{a}_s)$.
For $\mathcal{A}_1=\mathcal{A}_2=\dots=\mathcal{A}_n$ and $\mathfrak{a}_1=\mathfrak{a}_2=\dots=\mathfrak{a}_s=\mathfrak{a}$, we can describe $\mathcal{A}_n(\mathcal{A},\mathfrak{a})=B(\mathcal{A},\mathcal{A},\dots,\mathcal{A};\mathfrak{a},\mathfrak{a},
\dots,\mathfrak{a})$.
In particular, let $B_n=\mathcal{A}_n(\mathcal{P}_3,\mathfrak{a})$, with $\deg_{\mathcal{P}_3}(\mathfrak{a})=2$ and $T_{n,k}=\mathcal{A}_n(\mathcal{C}_k,\mathfrak{a})$ be the types of bridge graphs.
Then we have $B_k=\mathcal{P}_k\circ \overline{\mathcal{P}}_2$, $T_{k,3}=\mathcal{P}_k\circ \mathcal{P}_2$ and $J_{j,k+1}= \mathcal{P}_j\circ \mathcal{C}_k$.

By using Theorem \ref{theoremcor}, we have the next results:
\begin{enumerate}
  \item $Mo(B_k)=3\left\lfloor\dfrac{(k-1)^2}{2}\right\rfloor+2k(3k-1).$
  \item $Mo(T_{k,3})=3\left\lfloor\dfrac{(k-1)^2}{2}\right\rfloor+2k(3k-1).$
  \item $Mo(J_{j,k+1})=(k+1)\left\lfloor\dfrac{(j-1)^2}{2}\right\rfloor+jk\left|2-j-jk\right|+2jk.$
\end{enumerate}
\end{example}
\subsection{Cartesian Product}
Here we denote the Cartesian product of $\mathcal{G}_l$ and $\mathcal{G}_m$ graphs with $\mathcal{G}_l\otimes \mathcal{G}_m,$ it has $\mathsf{V}(\mathcal{G}_l)\times \mathsf{V}(\mathcal{G}_m)$ vertex set and $(\mathfrak{u}_l,\mathfrak{u}_m)(\mathfrak{u'}_l,\mathfrak{u'}_m)\in \mathsf{E}(\mathcal{G}_l\otimes \mathcal{G}_m)$ if $\mathfrak{u}_l=\mathfrak{u'}_l$ and $\mathfrak{u}_m\mathfrak{u'}_m\in \mathsf{E}(\mathcal{G}_m)$, or $\mathfrak{u}_l\mathfrak{u'}_l \in \mathsf{E}(\mathcal{G}_l)$ and $\mathfrak{u}_m=\mathfrak{u'}_m$.

Now, we give the expression for Mostar index of $\mathcal{G}_{l_1}\otimes \mathcal{G}_{l_2}\otimes\dots\otimes G_{l_k}$ in the form of factor graphs.
\begin{theorem}\label{theoremcar}
Let $\mathcal{G}_{l_1},\mathcal{G}_{l_2},\dots,\mathcal{G}_{l_k}$ be graphs with $|\mathsf{V}(\mathcal{G}_{l_m})|=s_m$, $|\mathsf{E}(\mathcal{G}_{l_m})|=t_m$, $1\leq m\leq k$, and $s=\prod\limits_{m=1}^{k}s_m$.
Then we have
\begin{eqnarray*}
\begin{split}
Mo(\mathcal{G}_{l_1}\otimes\mathcal{G}_{l_2}\dots\mathcal{G}_{l_k})&= \sum\limits_{i=1}^{k+1}Mo(\mathcal{G}_{l_i}) \prod\limits_{j=1,i\neq j}^{k}s_j^2.
\end{split}
\end{eqnarray*}
\end{theorem}
\begin{proof}
Observe that $\mathrm{n}_{(\mathfrak{u}_{l_1},\mathfrak{u}_{l_2})}(\mathfrak{e}=(\mathfrak{u}_{l_1},\mathfrak{u}_{l_2})
(\mathfrak{u}_{l_1},\mathfrak{u'}_{l_2})|\mathcal{G}_{l_1}\otimes \mathcal{G}_{l_2})=s_1\mathrm{n}_{\mathfrak{u}_{l_2}}(\mathfrak{e}=\mathfrak{u}_{l_2}\mathfrak{u'}_{l_2}
|\mathcal{G}_{l_2})$ and $\mathrm{n}_{(\mathfrak{u}_{l_1},\mathfrak{u'}_{l_2})}(\mathfrak{e}=(\mathfrak{u}_{l_1},\mathfrak{u}_{l_2})
(\mathfrak{u}_{l_1},\mathfrak{u'}_{l_2})|\mathcal{G}_{l_1}\otimes \mathcal{G}_{l_2})= s_1\mathrm{n}_{\mathfrak{u'}_{l_2}}(\mathfrak{e}=\mathfrak{u}_{l_2}\mathfrak{u'}_{l_2}|\mathcal{G}_{l_2})$.
Analogous $\mathrm{n}_{(\mathfrak{u}_{l_1},\mathfrak{u}_{l_2})}(\mathfrak{e}=(\mathfrak{u}_{l_1},\mathfrak{u}_{l_2})
(\mathfrak{u'}_{l_1},\mathfrak{u}_{l_2})|\mathcal{G}_{l_1}\otimes \mathcal{G}_{l_2})=s_2\mathrm{n}_{\mathfrak{u}_{l_1}}(\mathfrak{e}=\mathfrak{u}_{l_1}\mathfrak{u'}_{l_1}|
\mathcal{G}_{l_1})$ and $\mathrm{n}_{(\mathfrak{u'}_{l_1},\mathfrak{u}_{l_2})}(\mathfrak{e}=(\mathfrak{u}_{l_1},\mathfrak{u}_{l_2})
(\mathfrak{u'}_{l_1},\mathfrak{u}_{l_2})|\mathcal{G}_{l_1}\otimes \mathcal{G}_{l_2})= s_2\mathrm{n}_{\mathfrak{u'}_{l_1}}(\mathfrak{e}=\mathfrak{u}_{l_1}\mathfrak{u'}_{l_1}|\mathcal{G}_{l_1})$.
Thus
\begin{eqnarray}\label{car}
\begin{split}
Mo(\mathcal{G}_{l_1}\otimes \mathcal{G}_{l_2})&=\sum\limits_{\mathfrak{u}_{l_1}\in \mathsf{V}(\mathcal{G}_{l_1})}\sum\limits_{\mathfrak{u}_{l_2}\mathfrak{u'}_{l_2}\in \mathsf{E}(\mathcal{G}_{l_2})}|\mathrm{n}_{(\mathfrak{u}_{l_1},\mathfrak{u}_{l_2})}
(\mathfrak{e}=(\mathfrak{u}_{l_1},\mathfrak{u}_{l_2})(\mathfrak{u}_{l_1},\mathfrak{u'}_{l_2})|\mathcal{G}_{l_1}
\otimes \mathcal{G}_{l_2})
\\
&-\mathrm{n}_{(\mathfrak{u}_{l_1},\mathfrak{u'}_{l_2})}(\mathfrak{e}=(\mathfrak{u}_{l_1},
\mathfrak{u}_{l_2})(\mathfrak{u}_{l_1},\mathfrak{u'}_{l_2})|\mathcal{G}_{l_1}\otimes \mathcal{G}_{l_2})|
\\
&+\sum\limits_{\mathfrak{u}_{l_2}\in \mathsf{V}(\mathcal{G}_{l_2})}\sum\limits_{\mathfrak{u}_{l_1}\mathfrak{u'}_{l_1}\in \mathsf{E}(\mathcal{G}_{l_1})}|\mathrm{n}_{(\mathfrak{u}_{l_1},\mathfrak{u}_{l_2})}(\mathfrak{e}=
(\mathfrak{u}_{l_1},\mathfrak{u}_{l_2})(\mathfrak{u'}_{l_1},\mathfrak{u}_{l_2})|\mathcal{G}_{l_1}\otimes \mathcal{G}_{l_2})
\\
&-\mathrm{n}_{(\mathfrak{u'}_{l_1},\mathfrak{u}_{l_2})}(\mathfrak{e}=(\mathfrak{u}_{l_1},\mathfrak{u}_{l_2})
(\mathfrak{u'}_{l_1},\mathfrak{u}_{l_2})|\mathcal{G}_{l_1}\otimes \mathcal{G}_{l_2})|
\\
&=\sum\limits_{\mathfrak{u}_{l_1}\in \mathsf{V}(\mathcal{G}_{l_1})}\sum\limits_{\mathfrak{u}_{l_2}\mathfrak{u'}_{l_2}\in \mathsf{E}(\mathcal{G}_{l_2})}|s_1\mathrm{n}_{\mathfrak{u}_{l_2}}(\mathfrak{e}=\mathfrak{u}_{l_2}
\mathfrak{u'}_{l_2}|\mathcal{G}_{l_2})-s_1\mathrm{n}_{\mathfrak{u'}_{l_2}}(\mathfrak{e}=\mathfrak{u}_{l_2}
\mathfrak{u'}_{l_2}|\mathcal{G}_{l_2})|
\\
&+\sum\limits_{\mathfrak{u}_{l_2}\in \mathsf{V}(\mathcal{G}_{l_2})}\sum\limits_{\mathfrak{u}_{l_1}\mathfrak{u'}_{l_1}\in \mathsf{E}(\mathcal{G}_{l_1})}\left|s_2\mathrm{n}_{\mathfrak{u}_{l_1}}(\mathfrak{e}=\mathfrak{u}_{l_1}
\mathfrak{u'}_{l_1}|\mathcal{G}_{l_1})-s_2\mathrm{n}_{\mathfrak{u'}_{l_1}}(\mathfrak{e}=\mathfrak{u}_{l_1}
\mathfrak{u'}_{l_1}|\mathcal{G}_{l_1})\right|
\\
&=s_1^2Mo(\mathcal{G}_{l_2})+s_2^2Mo(\mathcal{G}_{l_1})
\end{split}
\end{eqnarray}
Use induction on $k$.
By above \eqref{car}, the result is valid for $k=2$.
Let $k\geq3$ and assume the theorem holds for $k$.
Set $\mathbb{G}=\mathcal{G}_{l_1}\otimes \mathcal{G}_{l_2}\dots\otimes \mathcal{G}_{l_k}$.
Then we have
\begin{eqnarray*}
\begin{split}
Mo(\mathcal{G}_{l_1}\otimes \mathcal{G}_{l_2}\dots\otimes \mathcal{G}_{l_k})&=Mo(\mathbb{G}\otimes \mathcal{G}_{k+1})
\\
&=Mo(\mathbb{G})s^2_{k+1}+Mo(\mathcal{G}_{k+1})|\mathsf{V}(\mathcal{G})|^2
\\
&=s^2_{k+1}\sum\limits_{i=1}^k Mo(\mathcal{G}_i)\prod\limits_{j=1,i\neq j}^k s^2_j+Mo(\mathcal{G}_{k+1})\prod\limits_{i=1}^k s^2_i
\\
&=\sum\limits_{i=1}^k Mo(\mathcal{G}_i)\prod\limits_{j=1,i\neq j}^{k+1}s^2_j
+Mo(\mathcal{G}_{k+1})\prod\limits_{i=1}^k s^2_i
\\
&=\sum\limits_{i=1}^{k+1} Mo(\mathcal{G}_i)\prod\limits_{j=1,i\neq j}^{k+1}s^2_j
\end{split}
\end{eqnarray*}
This finishes the proof.
\end{proof}

By the use of Theorem \ref{theoremcar}, we can find the Mostar index of $k$-th Cartesian power of a graph $\mathbb{G}$.
\begin{cor}
For the positive integer $k$, $\mathbb{G}$ is a graph of order $s$, then
\begin{equation*}
Mo(\mathbb{G}^k)=ks^{2(k-1)}Mo(\mathbb{G}).
\end{equation*}
\end{cor}
\begin{example}
Let $\mathbb{S}=\mathcal{C}_a\otimes \mathcal{C}_b$ and $\mathbb{R}=\mathcal{P}_a\otimes \mathcal{C}_b$, for some integers $a,b\geq 3$, denote a $\mathcal{C}_4$-nanotorus and $\mathcal{C}_4$-nanotube, respectively. Then by using Theorem \ref{theoremcar} and Proposition \ref{pro}, we obtain
\begin{enumerate}
  \item $Mo(\mathbb{S})=Mo(\mathcal{C}_a\otimes \mathcal{C}_b)=0$.
  \item $Mo(\mathbb{R})=Mo(\mathcal{P}_a\otimes \mathcal{C}_b)=b^2\left\lfloor\dfrac{(a-1)^2}{2}\right\rfloor.$
\end{enumerate}
\end{example}

\begin{example}
Consider the rectangular grid $\mathbb{G}=\mathcal{P}_a\otimes \mathcal{P}_b$, shown in Figure \ref{grid}.
By using Theorem \ref{theoremcar} and Proposition \ref{pro}, we obtain
\begin{equation*}
Mo(\mathbb{G})=Mo(\mathcal{P}_a\otimes \mathcal{P}_b)=a^2\left\lfloor\dfrac{(b-1)^2}{2}\right\rfloor+b^2\left\lfloor\dfrac{(a-1)^2}{2}\right\rfloor.
\end{equation*}
\end{example}
The graph $\mathcal{P}_2\otimes \mathcal{P}_{a+1}$ constructed by $a$ squares is said to be the ladder graph with $2a+2$ vertices and represented by $\mathbb{L}_a$ .
This graph is also the molecular graph, which can be related to the polyomino structures and known as the linear polyomino chain.
\begin{example}
Consider the ladder graph $\mathbb{L}_a=\mathcal{P}_2\otimes \mathcal{P}_{a+1}$.
By using Theorem \ref{theoremcar} and Proposition \ref{pro}, we obtain
\begin{equation*}
Mo(\mathbb{G})=Mo(\mathcal{P}_2\otimes \mathcal{P}_{a+1})=4\left\lfloor\dfrac{a^2}{2}\right\rfloor.
\end{equation*}
\end{example}

\begin{example}
The Hamming graph is a connected graph with $k$-tuples vertices $h_1h_2\dots h_k$ where $h_i\in \{0,1,\dots,s_{i-1}\}$, $s_i\geq2$, let two vertices be adjacent if the corresponding tuples differ in precisely one place. It is usually denoted as $\mathbb{H}_{s_1,s_2,\dots,s_k}=\bigotimes\limits_{l=1}^{k}\mathcal{K}_{s_l}$.
By using Theorem \ref{theoremcar} and then Proposition \ref{pro}, we get $Mo(\mathbb{H}_{s_1,s_2,\dots,s_k})=0$ such that $Mo(\mathcal{K}_{s_i})=0$.

If $h_1=h_2=\dots=h_k=2$, then the Hamming graph will be a $k$-dimensional hypercube graph, and it is denoted by $\mathcal{Q}_k$.
Then $Mo(\mathcal{Q}_k)=0$.
\end{example}


\subsection{Join of graphs}
The join of $\mathcal{G}_l$ and $\mathcal{G}_m$ graphs is denoted as $\mathcal{G}_l+\mathcal{G}_m$, which consists of edge sets $\mathsf{E}(\mathcal{G}_l)$ and $\mathsf{E}(\mathcal{G}_m)$, and disjoint vertex sets $\mathsf{V}(\mathcal{G}_l)$ and $\mathsf{V}(\mathcal{G}_m)$.
It is graph union $\mathcal{G}_l\cup \mathcal{G}_m$ including all the links joining the elements of $\mathsf{V}(\mathcal{G}_l)$ and $\mathsf{V}(\mathcal{G}_m)$.
For $\mathcal{G}_m =\underbrace{\mathcal{G}_l+\mathcal{G}_l+\dots+\mathcal{G}_l}_{k\ times}$, we represent $\mathcal{G}_m$ by $k\mathcal{G}_l$.

Next, we calculate the Mostar index of $\mathcal{G}_l+\mathcal{G}_m$.
\begin{theorem}\label{theoremjoin}
Let $\mathcal{G}_l$ and $\mathcal{G}_m$ be the two graphs. Then
\begin{eqnarray*}
\begin{split}
Mo(\mathcal{G}_l+\mathcal{G}_m)&\leq irr(\mathcal{G}_l)+irr(\mathcal{G}_m)+s_1s_2\left|s_2-s_1\right|+2(s_2t_1+s_1t_2).
\end{split}
\end{eqnarray*}
\end{theorem}
\begin{proof}
For an edge $\mathfrak{u}_l\mathfrak{u'}_l$ of a graph $\mathcal{G}_l$, let $N_{\mathcal{G}_l}(\mathfrak{u}_l\mathfrak{u'}_l)$ be the set of common neighbors of $\mathfrak{u}_l$ and $\mathfrak{u'}_l$.
\begin{eqnarray*}
\begin{split}
Mo(\mathcal{G}_l+\mathcal{G}_m)&=\sum\limits_{\mathfrak{u}_l\mathfrak{u'}_l\in \mathsf{E}(\mathcal{G}_l)}\left|\mathrm{n}_{\mathfrak{u}_l}(\mathfrak{u}_l\mathfrak{u'}_l|\mathcal{G}_l
+\mathcal{G}_m)-\mathrm{n}_{\mathfrak{u'}_l}(\mathfrak{u}_l\mathfrak{u'}_l|\mathcal{G}_l+\mathcal{G}_m)\right|
\\
&+\sum\limits_{\mathfrak{u}_m\mathfrak{u'}_m\in \mathsf{E}(\mathcal{G}_m)}\left|\mathrm{n}_{\mathfrak{u}_m}(\mathfrak{u}_m\mathfrak{u'}_m|
\mathcal{G}_l+\mathcal{G}_m)-\mathrm{n}_{\mathfrak{u'}_m}(\mathfrak{u}_m\mathfrak{u'}_m|
\mathcal{G}_l+\mathcal{G}_m)\right|
\\
&+\sum\limits_{\mathfrak{u}_l\in \mathsf{V}(\mathcal{G}_l)}\sum\limits_{\mathfrak{u}_m\in \mathsf{V}(\mathcal{G}_m)}\left|\mathrm{n}_{\mathfrak{u}_l}(\mathfrak{u}_l\mathfrak{u}_m|
\mathcal{G}_l+\mathcal{G}_m)-\mathrm{n}_{\mathfrak{u}_m}(\mathfrak{u}_l\mathfrak{u}_m|
\mathcal{G}_l+\mathcal{G}_m)\right|.
\end{split}
\end{eqnarray*}
Since the join of two graphs has diameter at most two.
Observe that if $\mathfrak{u}_l\mathfrak{u'}_l\in \mathsf{E}(\mathcal{G}_l)$ then we have $\mathrm{n}_{\mathfrak{u}_l}(\mathfrak{e}=\mathfrak{u}_l\mathfrak{u'}_l|\mathcal{G}_l+\mathcal{G}_m)
=\deg_{\mathcal{G}_l}(\mathfrak{u}_l)-|N_{\mathcal{G}_l}(\mathfrak{u}_l\mathfrak{u'}_l)|$ and $\mathrm{n}_{\mathfrak{u'}_l}(\mathfrak{e}=\mathfrak{u}_l\mathfrak{u'}_l|\mathcal{G}_l+\mathcal{G}_m)
=\deg_{\mathcal{G}_l}(\mathfrak{u'}_l)-|N_{\mathcal{G}_l}(\mathfrak{u}_l\mathfrak{u'}_l)|$.
Similarly if $\mathfrak{u}_m\mathfrak{u'}_m\in \mathsf{E}(\mathcal{G}_m)$ then we have $\mathrm{n}_{\mathfrak{u}_m}(\mathfrak{e}=\mathfrak{u}_m\mathfrak{u'}_m|\mathcal{G}_l+\mathcal{G}_m)
=\deg_{\mathcal{G}_m}(\mathfrak{u}_m)-|N_{\mathcal{G}_m}(\mathfrak{u}_m\mathfrak{u'}_m)|$ and $\mathrm{n}_{\mathfrak{u'}_m}(\mathfrak{e}=\mathfrak{u}_m\mathfrak{u'}_m|\mathcal{G}_l+\mathcal{G}_m)
=\deg_{\mathcal{G}_m}(\mathfrak{u'}_m)-|N_{\mathcal{G}_m}(\mathfrak{u}_m\mathfrak{u'}_m)|$.
Analogous if $\mathfrak{u}_l\in \mathsf{V}(\mathcal{G}_l)$ and $\mathfrak{u}_m\in \mathsf{V}(\mathcal{G}_m)$ then we have $\mathrm{n}_{\mathfrak{u}_l}(\mathfrak{e}=\mathfrak{u}_l\mathfrak{u}_m|\mathcal{G}_l+\mathcal{G}_m)
=s_2-\deg_{\mathcal{G}_m}(\mathfrak{u}_m)$ and $\mathrm{n}_{\mathfrak{u}_m}(\mathrm{e}=\mathfrak{u}_l\mathfrak{u}_m|\mathcal{G}_l+\mathcal{G}_m)
=s_1-\deg_{\mathcal{G}_l}(\mathfrak{u}_l)$.
Therefore
\begin{eqnarray*}
\begin{split}
Mo(\mathcal{G}_l+\mathcal{G}_m)&=\sum\limits_{\mathfrak{u}_l\mathfrak{u'}_l\in \mathsf{E}(\mathcal{G}_l)}\left|\deg_{\mathcal{G}_l}(\mathfrak{u}_l)-|N_{\mathcal{G}_l}(\mathfrak{u}_l
\mathfrak{u'}_l)|-\deg_{\mathcal{G}_l}(\mathfrak{u'}_l)+|N_{\mathcal{G}_l}(\mathfrak{u}_l\mathfrak{u'}_l)|\right|
\\
&+\sum\limits_{\mathfrak{u}_m\mathfrak{u'}_m\in \mathsf{E}(\mathcal{G}_m)}\left|\deg_{\mathcal{G}_m}(\mathfrak{u}_m)-|N_{\mathcal{G}_m}(\mathfrak{u}_m
\mathfrak{u'}_m)|-\deg_{\mathcal{G}_m}(\mathfrak{u'}_m)+|N_{\mathcal{G}_m}(\mathfrak{u}_m\mathfrak{u'}_m)|\right|
\\
&+\sum\limits_{\mathfrak{u}_l\in \mathsf{V}(\mathcal{G}_l)}\sum\limits_{\mathfrak{u}_m\in \mathsf{V}(\mathcal{G}_m)}\left|s_2-\deg_{\mathcal{G}_m}(\mathfrak{u}_m)-s_1+\deg_{\mathcal{G}_l}
(\mathfrak{u}_l)\right|
\\
&\leq\sum\limits_{\mathfrak{u}_l\mathfrak{u'}_l\in \mathsf{E}(\mathcal{G}_l)}\left|\deg_{\mathcal{G}_l}(\mathfrak{u}_l)-\deg_{\mathcal{G}_l}(\mathfrak{u'}_l)\right|
+\sum\limits_{\mathfrak{u}_m\mathfrak{u'}_m\in \mathsf{E}(\mathcal{G}_m)}\left|\deg_{\mathcal{G}_m}(\mathfrak{u}_m)-\deg_{\mathcal{G}_m}(\mathfrak{u'}_m)\right|
\\
&+\sum\limits_{\mathfrak{u}_l\in \mathsf{V}(\mathcal{G}_l)}\sum\limits_{\mathfrak{u}_m\in \mathsf{V}(\mathcal{G}_m)}\left|s_2-s_1\right|
+\sum\limits_{\mathfrak{u}_l\in \mathsf{V}(\mathcal{G}_l)}\sum\limits_{\mathfrak{u}_m\in \mathsf{V}(\mathcal{G}_m)}\deg_{\mathcal{G}_m}(\mathfrak{u}_m)+\sum\limits_{\mathfrak{u}_l\in \mathsf{V}(\mathcal{G}_l)}\sum\limits_{\mathfrak{u}_m\in \mathsf{V}(\mathcal{G}_m)}\deg_{\mathcal{G}_l}(\mathfrak{u}_l)
\\
&=irr(\mathcal{G}_l)+irr(\mathcal{G}_m)+s_1s_2\left|s_2-s_1\right|+2(s_2t_1+s_1t_2).
\end{split}
\end{eqnarray*}
This completes the proof.
\end{proof}
\begin{cor}\label{corjoin}
Let $\mathcal{G}_l$ and $\mathcal{G}_m$ be the $\mathfrak{r}_1$ and $\mathfrak{r}_2$ regular graphs, respectively.
Then
\begin{equation*}
Mo(\mathcal{G}_l+\mathcal{G}_m)=s_1s_2\left|s_2-s_1+\mathfrak{r}_1-\mathfrak{r}_2\right|.
\end{equation*}
\end{cor}
\begin{example}
The cone graph $\mathcal{C}_{f,g}$ can be expressed as $\mathcal{C}_f +\overline{\mathcal{K}}_g$ and its Mostar index is $Mo(\mathcal{C}_{f,g})=fg|f-g+2|$.
\end{example}

For a given graph $\mathcal{G}_m$, the suspension of $\mathcal{G}_m$ is described as $\mathcal{K}_1+\mathcal{G}_m$.
The next result can be deduced from the Corollary \ref{corjoin}.
\begin{example}
For a graph $\mathcal{G}_m$ with $|\mathsf{V}(\mathcal{G}_m)|=s$, we have
\begin{equation}\label{q1}
Mo(\mathcal{K}_1+\mathcal{G}_m)\leq irr(\mathcal{G}_m)+s(s-1)+2s.
\end{equation}
If $\mathcal{G}_m$ is a $\mathfrak{r}$-regular graph then $Mo(\mathcal{K}_1+\mathcal{G}_m)=s|s-1-\mathfrak{r}|$.
Star graph $S_{s+1}$, fan graph $F_{s+1}$ and wheel graph $W_{s+1}$ on $s+1$ vertices are the suspensions of the empty graph, $\overline{\mathcal{K}_s}$, $\mathcal{P}_s$ and $\mathcal{C}_s$, respectively.
Then, by \eqref{q1}, we have
\begin{eqnarray*}
\begin{tabular}{c c c}
$Mo(S_{s+1})=s(s-1)$,& $Mo(F_{s+1})\leq s(s+1)$, & $Mo(W_{s+1})=s|s-3|.$
\end{tabular}
\end{eqnarray*}
\end{example}
\begin{example}
The flower graph or dutch windmill graph is the suspension of $g$ copies of $\mathcal{K}_2$, denoted by $g\mathcal{K}_2$.
The Mostar index of flower graph is given by $Mo(\mathcal{K}_1+g\mathcal{K}_2)\leq4g$.
\end{example}
\subsection{Lexicographic Product}
The lexicographic product of $\mathsf{G}_l$ and $\mathsf{G}_m$ graphs is represented by $\mathsf{G}_l[\mathsf{G}_m].$
It has $\mathsf{V}(\mathsf{G}_l)\times \mathsf{V}(\mathsf{G}_m)$ vertex set and $(\mathfrak{u}_l,\mathfrak{u}_m)(\mathfrak{u'}_l,\mathfrak{u'}_m)\in \mathsf{E}(\mathsf{G}_l[\mathsf{G}_m])$ if $g_1g_2\in \mathsf{E}(\mathsf{G}_l)$ or $\mathfrak{u}_l=\mathfrak{u'}_l$ and $\mathfrak{u}_m\mathfrak{u'}_m\in \mathsf{E}(\mathsf{G}_m)$.

Now, we give the expression for Mostar index of lexicographic product of $\mathsf{G}_l$ and $\mathsf{G}_m$.
\begin{theorem}\label{theoremcom}
Let $\mathsf{G}_l$ and $\mathsf{G}_m$ be the two graphs. Then
\begin{eqnarray*}
\begin{split}
Mo(\mathsf{G}_l[\mathsf{G}_m])&\leq \left\{
\begin{array}{ll}
s_2^3Mo(\mathsf{G}_l)+s_1irr(\mathsf{G}_m)+\dfrac{t_1}{6}\left(2s_2^3-3s_2^2-2s_2+3\right)        & \mbox{if $s_2$ is odd,} \\\\
s_2^3Mo(\mathsf{G}_l)+s_1irr(\mathsf{G}_m)+\dfrac{t_1}{6}\left(2s_2^3-3s_2^2-2s_2\right)          & \mbox{if $s_2$ is even.}
\end{array}
\right.
\end{split}
\end{eqnarray*}
\end{theorem}
\begin{proof}
For an edge $\mathfrak{u}_l\mathfrak{u'}_l$ of a graph $\mathsf{G}_l$, let $N_{\mathsf{G}_l}(\mathfrak{u}_l\mathfrak{u'}_l)$ be the set of common neighbors of $\mathfrak{u}_l$ and $\mathfrak{u'}_l$.
Observe that if $\mathfrak{u}_l\in \mathsf{V}(\mathsf{G}_l)$ and $\mathfrak{u}_m\mathfrak{u'_m}\in \mathsf{E}(\mathsf{G}_m)$ then we have $\mathrm{n}_{(\mathfrak{u}_l,\mathfrak{u}_m)}(\mathfrak{e}=(\mathfrak{u}_l,\mathfrak{u}_m)(\mathfrak{u}_l,
\mathfrak{u'}_m)|\mathsf{G}_l[\mathsf{G}_m])=\deg_{\mathsf{G}_m}(\mathfrak{u}_m)
-|N_{\mathsf{G}_m}(\mathfrak{u}_m\mathfrak{u'}_m)|$ and $\mathrm{n}_{(\mathfrak{u}_l,\mathfrak{u'}_m)}(\mathrm{e}=(\mathfrak{u}_l,\mathfrak{u}_m)(\mathfrak{u}_l,
\mathfrak{u'}_m)|\mathsf{G}_l[\mathsf{G}_m])=\deg_{\mathsf{G}_m}(\mathfrak{u}_m)-|N_{\mathsf{G}_m}
(\mathfrak{u}_m\mathfrak{u'}_m)|$.
Analogous if $\mathfrak{u}_m,\mathfrak{u'}_m\in \mathsf{V}(\mathsf{G}_m)$ and $\mathfrak{u}_l\mathfrak{u'}_l\in \mathsf{E}(\mathsf{G}_l)$ then we have $\mathrm{n}_{(\mathfrak{u}_l,\mathfrak{u}_m)}(\mathfrak{e}=(\mathfrak{u}_l,\mathfrak{u}_m)(\mathfrak{u'}_l,
\mathfrak{u'}_m)|\mathsf{G}_l[\mathsf{G}_m])=s_2-\deg_{\mathsf{G}_m}(\mathfrak{u'}_m)+s_2 \mathrm{n}_{\mathfrak{u}_l}(\mathfrak{u}_l\mathfrak{u'}_l|\mathsf{G}_l)$ and $\mathrm{n}_{(\mathfrak{u'}_l,\mathfrak{u}_m)}(\mathfrak{e}=(\mathfrak{u}_l,\mathfrak{u}_m)
(\mathfrak{u'}_l,\mathfrak{u}_m)|\mathsf{G}_l[\mathsf{G}_m])=s_2-\deg_{\mathsf{G}_m}(\mathfrak{u}_m)+s_2 \mathrm{n}_{\mathfrak{u'}_l}(\mathfrak{u}_l\mathfrak{u'}_l|\mathsf{G}_l)$.
Therefore
\begin{eqnarray*}
\begin{split}
Mo(\mathsf{G}_l[\mathsf{G}_m])&=\sum\limits_{\mathfrak{u}_l\in \mathsf{V}(\mathsf{G}_l)}\sum\limits_{\mathfrak{u}_m\mathfrak{u'}_m\in \mathsf{E}(\mathsf{G}_m)}|\mathrm{n}_{(\mathfrak{u}_l,\mathfrak{u}_m)}(\mathfrak{e}=(\mathfrak{u}_l,
\mathfrak{u}_m)(\mathfrak{u}_l,\mathfrak{u'}_m)|\mathsf{G}_l[\mathsf{G}_m])
\\
&-\mathrm{n}_{(\mathfrak{u}_l,
\mathfrak{u'}_m)}(\mathfrak{e}=(\mathfrak{u}_l,\mathfrak{u}_m)(\mathfrak{u}_l,\mathfrak{u'}_m)
|\mathsf{G}_l[\mathsf{G}_m]))|
\\
&+\sum\limits_{\mathfrak{u}_m\in \mathsf{V}(\mathsf{G}_m)}\sum\limits_{\mathfrak{u'}_m\in \mathsf{V}(\mathsf{G}_m)}\sum\limits_{\mathfrak{u}_l\mathfrak{u'}_l\in \mathsf{E}(\mathsf{G}_l)}|\mathrm{n}_{(\mathfrak{u}_l,\mathfrak{u}_m)}(\mathfrak{e}=(\mathfrak{u}_l,
\mathfrak{u}_m)(\mathfrak{u'}_l,\mathfrak{u'}_m)|\mathsf{G}_l[\mathsf{G}_m])
\\
&-\mathrm{n}_{(\mathfrak{u}_l,
\mathfrak{u'}_m)}(\mathfrak{e}=(\mathfrak{u}_l,\mathfrak{u}_m)(\mathfrak{u'}_l,\mathfrak{u'}_m)
|\mathsf{G}_l[\mathsf{G}_m])|
\\
&=\sum\limits_{\mathfrak{u}_l\in \mathsf{V}(\mathsf{G}_l)}\sum\limits_{\mathfrak{u}_m\mathfrak{u'}_m\in \mathsf{E}(\mathsf{G}_m)}\left|\deg_{\mathsf{G}_l}(\mathfrak{u}_m)-|N_{\mathsf{G}_m}(\mathfrak{u}_m\mathfrak{u'}_m)|
-\deg_{\mathsf{G}_m}(\mathfrak{u'}_m)+|N_{\mathsf{G}_m}(\mathfrak{u}_m\mathfrak{u'}_m)|\right|
\\
&+\sum\limits_{\mathfrak{u}_m\in \mathsf{V}(\mathsf{G}_m)}\sum\limits_{\mathfrak{u'}_m\in \mathsf{V}(\mathsf{G}_m)}\sum\limits_{\mathfrak{u}_l\mathfrak{u'}_l\in \mathsf{E}(\mathsf{G}_l)}|s_2-\deg_{\mathsf{G}_m}(\mathfrak{u'}_m)+s_2 \mathrm{n}_{\mathfrak{u}_l}(\mathfrak{u}_l\mathfrak{u'}_l|\mathsf{G}_l)-s_2
+\deg_{\mathsf{G}_m}(\mathfrak{u}_m)
\\
&-s_2\mathrm{n}_{\mathfrak{u'}_l}(\mathfrak{u}_l\mathfrak{u'}_l|\mathsf{G}_l)|
\\
&\leq s_1 irr(\mathsf{G}_m)+s_2^3 Mo(\mathsf{G}_l)+2t_2 irr_t(\mathsf{G}_m).
\end{split}
\end{eqnarray*}
By using Proposition \ref{pro1}, we obtained the desired result. This finishes the proof.
\end{proof}
\begin{example}
The fence graph and closed fence graph are the lexicographic product of $\mathcal{P}_g$ and $\mathcal{P}_2$, and, $\mathcal{C}_g$ and $\mathcal{P}_2$ respectively.
Then from Theorem \ref{theoremcom} and Proposition \ref{pro}, we have $Mo(\mathcal{C}_g[\mathcal{P}_2])=0$ and $Mo(\mathcal{P}_g[\mathcal{P}_2])\leq8\left\lfloor\dfrac{(g-1)^2}{2}\right\rfloor$.
\end{example}
\begin{example}
Let $\mathcal{P}_g$ and $\mathcal{P}_h$ be the two paths. Then
\begin{eqnarray*}
\begin{split}
Mo(\mathcal{P}_g[\mathcal{P}_h])&\leq\left\{
\begin{array}{ll}
2g+h^3\left\lfloor\dfrac{(g-1)^2}{2}\right\rfloor+\dfrac{g-1}{6}(2h^3-3h^2-2h+3), & \mbox{if $h$ is odd,}\\\\
2g+h^3\left\lfloor\dfrac{(g-1)^2}{2}\right\rfloor+\dfrac{g-1}{6}(2h^3-3h^2-2h),   & \mbox{if $h$ is even.}
\end{array}
\right.
\end{split}
\end{eqnarray*}
\end{example}
\subsection{Indu-Bala Product}
The Indu-Bala product $\mathcal{G}_l\blacktriangledown \mathcal{G}_m$ is obtained from two disjoint copies of $\mathcal{G}_l+\mathcal{G}_m$ by joining the corresponding vertices in the two copies of $\mathcal{G}_m$. For example, (see Figure \ref{fence1}).
The order and size of $\mathcal{G}_l\blacktriangledown \mathcal{G}_m$ are $2(s_{1}+s_{2})$ and $2t_1+2t_2+2s_1s_2+s_2$, respectively.

Now, we give the expression for Mostar index of Indu-Bala product of $\mathcal{G}_l$ and $\mathcal{G}_m$.
\begin{theorem}\label{theoremindu}
Let $\mathcal{G}_l$ and $\mathcal{G}_m$ be the two graphs. Then
\begin{eqnarray*}
\begin{split}
Mo(\mathcal{G}_l\blacktriangledown \mathcal{G}_m)&\leq2\left(irr(\mathcal{G}_l)+2irr(\mathcal{G}_m)
+s_1s_2\left|s_2-2s_1-1\right|+2(s_2t_1+s_1t_2)\right).
\end{split}
\end{eqnarray*}
\end{theorem}
\begin{proof}
For an edge $\mathfrak{u}_l\mathfrak{u'}_l$ of a graph $\mathcal{G}_l$, let $N_{\mathcal{G}_l}(\mathfrak{u}_l\mathfrak{u'}_l)$ be the set of common neighbors of $\mathfrak{u}_l$ and $\mathfrak{u'}_l$.
\begin{eqnarray}\label{indu}
\begin{split}
Mo(\mathcal{G}_l\blacktriangledown \mathcal{G}_m)&=2\bigg(\sum\limits_{\mathfrak{u}_l\mathfrak{u'}_l\in \mathsf{E}(\mathcal{G}_l)}\left|\mathrm{n}_{\mathfrak{u}_l}(\mathfrak{u}_l\mathfrak{u'}_l|\mathcal{G}_l
\blacktriangledown \mathcal{G}_m)-\mathrm{n}_{\mathfrak{u'}_l}(\mathfrak{u}_l\mathfrak{u'}_l|\mathcal{G}_l\blacktriangledown \mathcal{G}_m)\right|
\\
&+\sum\limits_{\mathfrak{u}_m\mathfrak{u'}_m\in \mathsf{E}(\mathcal{G}_m)}\left|\mathrm{n}_{\mathfrak{u}_m}(\mathfrak{u}_m\mathfrak{u'}_m|\mathcal{G}_l
\blacktriangledown \mathcal{G}_m)-\mathrm{n}_{\mathfrak{u'}_m}(\mathfrak{u}_m\mathfrak{u'}_m|\mathcal{G}_l\blacktriangledown \mathcal{G}_m)\right|
\\
&+\sum\limits_{\mathfrak{u}_l\in \mathsf{V}(\mathcal{G}_l)}\sum\limits_{\mathfrak{u}_m\in \mathsf{V}(\mathcal{G}_m)}\left|\mathrm{n}_{\mathfrak{u}_l}(\mathfrak{u}_l\mathfrak{u}_m|\mathcal{G}_l
\blacktriangledown \mathcal{G}_m)-\mathrm{n}_{\mathfrak{u}_m}(\mathfrak{u}_l\mathfrak{u}_m|\mathcal{G}_l\blacktriangledown \mathcal{G}_m)\right|\bigg)
\\
&+\sum\limits_{\mathfrak{u}_m\in \mathsf{V}(\mathcal{G}_m)\mathfrak{u''}_m\in \mathsf{V}(\mathcal{G}_m)}
\left|\mathrm{n}_{\mathfrak{u}_m}(\mathfrak{u}_m\mathfrak{u''}_m|\mathcal{G}_l\blacktriangledown \mathcal{G}_m)-\mathrm{n}_{\mathfrak{u''}_m}(\mathfrak{u}_m\mathfrak{u''}_m|\mathcal{G}_l\blacktriangledown \mathcal{G}_m)\right|.
\end{split}
\end{eqnarray}
Since the Indu-Bala product of two graphs has diameter at most $3$.
Observe that if $\mathfrak{u}_l\mathfrak{u'}_l\in \mathsf{E}(\mathcal{G}_l)$ then we have $\mathrm{n}_{\mathfrak{u}_l}(\mathfrak{e}=\mathfrak{u}_l\mathfrak{u'}_l|\mathcal{G}_l\blacktriangledown \mathcal{G}_m)=\deg_{\mathcal{G}_l}(\mathfrak{u}_l)-|N_{\mathcal{G}_l}(\mathfrak{u}_l\mathfrak{u'}_l)|$ and $\mathrm{n}_{\mathfrak{u'}_l}(\mathfrak{e}=\mathfrak{u}_l\mathfrak{u'}_l|\mathcal{G}_l\blacktriangledown \mathcal{G}_m)=\deg_{\mathcal{G}_l}(\mathfrak{u'}_l)-|N_{\mathcal{G}_l}(\mathfrak{u}_l\mathfrak{u'}_l)|$.
Analogous if $\mathfrak{u}_m\mathfrak{u'}_m\in \mathsf{E}(\mathcal{G}_m)$ then we have $\mathrm{n}_{\mathfrak{u}_m}(\mathfrak{e}=\mathfrak{u}_m\mathfrak{u'}_m|\mathcal{G}_l\blacktriangledown \mathcal{G}_m)=2\deg_{\mathcal{G}_m}(\mathfrak{u}_m)-2|N_{\mathcal{G}_m}(\mathfrak{u}_m\mathfrak{u'}_m)|$ and $\mathrm{n}_{\mathfrak{u'}_m}(\mathfrak{e}=\mathfrak{u}_m\mathfrak{u'}_m|\mathcal{G}_l\blacktriangledown \mathcal{G}_m)=2\deg_{\mathcal{G}_m}(\mathfrak{u'}_m)-2|N_{\mathcal{G}_m}(\mathfrak{u}_m\mathfrak{u'}_m)|$.
Therefore
\begin{eqnarray}\label{indu1}
\begin{split}
&\sum\limits_{\mathfrak{u}_l\mathfrak{u'}_l\in \mathsf{E}(\mathcal{G}_l)}\left|\mathrm{n}_{\mathfrak{u}_l}(\mathfrak{u}_l\mathfrak{u'}_l|\mathcal{G}_l
\blacktriangledown \mathcal{G}_m)-\mathrm{n}_{\mathfrak{u'}_l}(\mathfrak{u}_l\mathfrak{u'}_l|\mathcal{G}_l\blacktriangledown \mathcal{G}_m)\right|
\\
=&\sum\limits_{\mathfrak{u}_l\mathfrak{u'}_l\in \mathsf{E}(\mathcal{G}_l)}\left|\deg_{\mathcal{G}_l}(\mathfrak{u}_l)-|N_{\mathcal{G}_l}(\mathfrak{u}_l
\mathfrak{u'}_l)|-\deg_{\mathcal{G}_l}(\mathfrak{u'}_l)+|N_{\mathcal{G}_l}(\mathfrak{u}_l\mathfrak{u'}_l)|\right|
\\
=&irr(\mathcal{G}_l).
\end{split}
\end{eqnarray}
Similarly
\begin{eqnarray}\label{indu2}
\begin{split}
&\sum\limits_{\mathfrak{u}_m\mathfrak{u'}_m\in \mathsf{E}(\mathcal{G}_m)}\left|\mathrm{n}_{\mathfrak{u}_m}(\mathfrak{u}_m\mathfrak{u'}_m|\mathcal{G}_l
\blacktriangledown \mathcal{G}_m)-\mathrm{n}_{\mathfrak{u'}_m}(\mathfrak{u}_m\mathfrak{u'}_m|\mathcal{G}_l\blacktriangledown \mathcal{G}_m)\right|
\\
=&\sum\limits_{\mathfrak{u}_m\mathfrak{u'}_m\in \mathsf{E}(\mathcal{G}_m)}\left|2\deg_{\mathcal{G}_m}(\mathfrak{u}_m)-2|N_{\mathcal{G}_m}(\mathfrak{u}_m
\mathfrak{u'}_m)|-2\deg_{\mathcal{G}_m}(\mathfrak{u'}_m)+2|N_{\mathcal{G}_m}(\mathfrak{u}_m\mathfrak{u'}_m)|\right|
\\
=&2irr(\mathcal{G}_m).
\end{split}
\end{eqnarray}

Also if $\mathfrak{u}_l\in \mathsf{V}(\mathcal{G}_l)$ and $\mathfrak{u}_m\in \mathsf{V}(\mathcal{G}_m)$ then we have $\mathrm{n}_{\mathfrak{u}_l}(\mathfrak{e}=\mathfrak{u}_l\mathfrak{u}_m|\mathcal{G}_l\blacktriangledown \mathcal{G}_m)=s_2-\deg_{\mathcal{G}_m}(\mathfrak{u}_m)$ and $\mathrm{n}_{\mathfrak{u}_m}(\mathfrak{e}=\mathfrak{u}_l\mathfrak{u}_m|\mathcal{G}_l\blacktriangledown \mathcal{G}_m)=2s_1-\deg_{\mathcal{G}_l}(\mathfrak{u}_l)+1$.
Therefore
\begin{eqnarray}\label{indu3}
\begin{split}
&\sum\limits_{\mathfrak{u}_l\in \mathsf{V}(\mathcal{G}_l)}\sum\limits_{\mathfrak{u}_m\in \mathsf{V}(\mathcal{G}_m)}\left|\mathrm{n}_{\mathfrak{u}_l}(\mathfrak{u}_l\mathfrak{u}_m|\mathcal{G}_l
\blacktriangledown \mathcal{G}_m)-\mathrm{n}_{\mathfrak{u}_m}(\mathfrak{u}_l\mathfrak{u}_m|\mathcal{G}_l\blacktriangledown \mathcal{G}_m)\right|
\\
=&\sum\limits_{\mathfrak{u}_l\in \mathsf{V}(\mathcal{G}_l)}\sum\limits_{\mathfrak{u}_m\in \mathsf{V}(\mathcal{G}_m)}\left|s_2-\deg_{\mathcal{G}_m}(\mathfrak{u}_m)-2s_1+\deg_{\mathcal{G}_l}(\mathfrak{u}_l)
-1\right|
\\
\leq&\sum\limits_{\mathfrak{u}_l\in \mathsf{V}(\mathcal{G}_l)}\sum\limits_{\mathfrak{u}_m\in \mathsf{V}(\mathcal{G}_m)}\left|s_2-2s_1-1\right|+\sum\limits_{\mathfrak{u}_l\in \mathsf{V}(\mathcal{G}_l)}\sum\limits_{\mathfrak{u}_m\in \mathsf{V}(\mathcal{G}_m)}\deg_{\mathcal{G}_m}(\mathfrak{u}_m)
+\sum\limits_{\mathfrak{u}_l\in \mathsf{V}(\mathcal{G}_l)}\sum\limits_{\mathfrak{u}_m\in \mathsf{V}(\mathcal{G}_m)}\deg_{\mathcal{G}_l}(\mathfrak{u}_l)
\\
=&s_1s_2\left|s_2-2s_1-1\right|+2(s_2t_1+s_1t_2).
\end{split}
\end{eqnarray}

Finally, if $\mathfrak{u}_m\in V(\mathcal{G}_m)$ and $\mathfrak{u''}_m$ in the copy of  $\mathcal{G}_m$ then we have $\mathrm{n}_{\mathfrak{u}_m}(\mathfrak{e}=\mathfrak{u}_m\mathfrak{u''}_m|\mathcal{G}_l\blacktriangledown \mathcal{G}_m)=\mathrm{n}_{\mathfrak{u''}_m}(\mathfrak{e}=\mathfrak{u}_m\mathfrak{u''}_m|\mathcal{G}_l
\blacktriangledown \mathcal{G}_m)$.
Therefore by using above results \eqref{indu1}-\eqref{indu3} in \eqref{indu}, we obtain the required expression for $Mo(\mathcal{G}_l\blacktriangledown \mathcal{G}_m)$.
\end{proof}

\begin{example}
\begin{enumerate}
  \item $Mo(\mathcal{P}_g\blacktriangledown \mathcal{P}_h)=2(6+gh(|h-2g-1|+4)-2(g+h)).$
  \item $Mo(\mathcal{P}_g\blacktriangledown \mathcal{C}_h)=2(2+gh(|h-2g-1|+4)-2h).$
  \item $Mo(\mathcal{C}_g\blacktriangledown \mathcal{P}_h)=2(4+gh|h-2g-1|+2g(2h-1)).$
\end{enumerate}
\end{example}
\section{Subdivision vertex-edge join of three graphs}
Very recently, a novel graph operation has been introduced by Wen et al. in \cite{31},  known as the subdivision vertex-edge join.
For a graph $\mathcal{G}_{l_1}$, $\mathsf{S}(\mathcal{G}_{l_1})$ denotes its subdividing graph, whose vertex set has two parts, one the primary vertices $\mathsf{V}(\mathcal{G}_{l_1})$, another, symbolized by $\mathsf{I}(\mathcal{G}_{l_1})$, the inserting vertices that are end points of $E(\mathcal{G}_{l_1})$.
Let $\mathcal{G}_{l_2}$ and $\mathcal{G}_{l_3}$ be the other two disjoint graphs.
The subdivision vertex-edge join of $\mathcal{G}_{l_1}$ with $\mathcal{G}_{l_2}$ and $\mathcal{G}_{l_3}$, denoted by $\mathcal{G}_{l_1}^{\mathsf{S}}\rhd(\mathcal{G}_{l_2}^{\mathsf{V}}\cup \mathcal{G}_{l_3}^{\mathsf{I}})$, is the graph consisting of of $\mathsf{S}(\mathcal{G}_{l_1})$, $\mathcal{G}_{l_2}$ and $\mathcal{G}_{l_3}$, all vertex-disjoint, and connecting the $g$-th vertex of $\mathsf{V}(\mathcal{G}_{l_1})$ to each vertex in $\mathsf{V}(\mathcal{G}_{l_2})$ and $g$-th vertex of $\mathsf{I}(\mathcal{G}_{l_1}$) to every vertex in $\mathsf{V}(\mathcal{G}_{l_3})$.
It can be saw that $\mathcal{G}_{l_1}^{\mathsf{S}}\rhd(\mathcal{G}_{l_2}^{\mathsf{V}}\cup \mathcal{G}_{l_3}^{\mathsf{I}})$ is $\mathcal{G}_{l_1}\dot{\vee}\mathcal{G}_{l_2}$ (is attained from $\mathsf{S}(\mathcal{G}_{l_1})$ and $\mathcal{G}_{l_2}$ by connecting each vertex of $\mathsf{V}(\mathcal{G}_{l_1})$ to every vertex of $\mathsf{V}(\mathcal{G}_{l_2})$ \cite{36}) if $\mathcal{G}_{l_3}$ is the null graph, and is $\mathcal{G}_{l_1}\underline{\vee}\mathcal{G}_{l_3}$ (is attained from $\mathsf{S}(\mathcal{G}_{l_1})$ and $\mathcal{G}_{l_3}$ by joining each vertex of $\mathsf{E}(\mathcal{G}_{l_1})$ to every vertex of $\mathsf{V}(\mathcal{G}_{l_3})$ \cite{32}) if $\mathcal{G}_{l_2}$ is the null graph.

Now, we give the expression for Mostar index of subdivision vertex-edge join of $\mathcal{G}_{l_1}$, $\mathcal{G}_{l_2}$ and $\mathcal{G}_{l_3}$.
\begin{theorem}
Let $G$ and $H$ be the two graphs. Then
\begin{eqnarray*}
\begin{split}
Mo(\mathcal{G}_{l_1}^\mathsf{S}\vartriangleright(\mathcal{G}_{l_2}^\mathsf{V}\cup \mathcal{G}_{l_3}^\mathsf{I}))&\leq irr(\mathcal{G}_{l_2})+irr(\mathcal{G}_{l_3})|+s_1s_2|s_2+s_3-s_1-t_1|+4t_1s_2+2s_1t_2
\\
&+t_1s_3|s_3+s_2-s_1+4|+2t_3t_1+s_1t_1|s_2+s_1-s_3-t_1-4|+4t_1^2.
\end{split}
\end{eqnarray*}
\end{theorem}
\begin{proof}
For an edge $\mathfrak{u}_l\mathfrak{u'}_l$ of a graph $\mathcal{G}_l$, let $N_{\mathcal{G}_l}(\mathfrak{u}_l\mathfrak{u'}_l)$ be the set of common neighbors of $\mathfrak{u}_l$ and $\mathfrak{u'}_l$.
\begin{eqnarray}\label{sve}
\begin{split}
&Mo(\mathcal{G}_{l_1}^\mathsf{S}\vartriangleright(\mathcal{G}_{l_2}^\mathsf{V}\cup \mathcal{G}_{l_3}^\mathsf{I}))\\
&=\sum\limits_{\mathfrak{u}_{l_2}\mathfrak{u'}_{l_2}\in \mathsf{E}(\mathcal{G}_{l_2})}\left|\mathrm{n}_{\mathfrak{u}_{l_2}}(\mathfrak{u}_{l_2}\mathfrak{u'}_{l_2}|
(\mathcal{G}_{l_1}^\mathsf{S}\vartriangleright(\mathcal{G}_{l_2}^\mathsf{V}\cup \mathcal{G}_{l_3}^\mathsf{I}))
-\mathrm{n}_{\mathfrak{u'}_{l_2}}(\mathfrak{u}_{l_2}\mathfrak{u'}_{l_2}|(\mathcal{G}_{l_1}^\mathsf{S}
\vartriangleright(\mathcal{G}_{l_2}^\mathsf{V}\cup \mathcal{G}_{l_3}^\mathsf{I}))\right|
\\
&+\sum\limits_{\mathfrak{u}_{l_3}\mathfrak{u'}_{l_3}\in \mathsf{E}(\mathcal{G}_{l_3})}\left|\mathrm{n}_{\mathfrak{u}_{l_3}}(\mathfrak{u}_{l_3}\mathfrak{u'}_{l_3}
|(\mathcal{G}_{l_1}^\mathsf{S}\vartriangleright(\mathcal{G}_{l_2}^\mathsf{V}\cup \mathcal{G}_{l_3}^\mathsf{I}))-\mathrm{n}_{\mathfrak{u'}_{l_3}}(\mathfrak{u}_{l_3}\mathfrak{u'}_{l_3}
|(\mathcal{G}_{l_1}^\mathsf{S}\vartriangleright(\mathcal{G}_{l_2}^\mathsf{V}\cup \mathcal{G}_{l_3}^\mathsf{I}))\right|
\\
&+\sum\limits_{\mathfrak{u}_{l_1}\in \mathsf{V}(\mathcal{G}_{l_1})}\sum\limits_{\mathfrak{u}_{l_2}\in \mathsf{V}(\mathcal{G}_{l_2})}\left|\mathrm{n}_{\mathfrak{u}_{l_1}}(\mathfrak{u}_{l_1}\mathfrak{u}_{l_2}
|(\mathcal{G}_{l_1}^\mathsf{S}\vartriangleright(\mathcal{G}_{l_2}^\mathsf{V}\cup \mathcal{G}_{l_3}^\mathsf{I}))-\mathrm{n}_{\mathfrak{u}_{l_2}}(\mathfrak{u}_{l_1}\mathfrak{u}_{l_2}
|(\mathcal{G}_{l_1}^\mathsf{S}\vartriangleright(\mathcal{G}_{l_2}^\mathsf{V}\cup \mathcal{G}_{l_3}^\mathsf{I}))\right|
\\
&+\sum\limits_{\mathfrak{u}_{l_1}\in \mathsf{E}(\mathcal{G}_{l_1})}\sum\limits_{\mathfrak{u}_{l_3}\in \mathsf{V}(\mathcal{G}_{l_3})}\left|\mathrm{n}_{\mathfrak{u}_{l_1}}(\mathfrak{u}_{l_1}\mathfrak{u}_{l_3}
|(\mathcal{G}_{l_1}^\mathsf{S}\vartriangleright(\mathcal{G}_{l_2}^\mathsf{V}\cup \mathcal{G}_{l_3}^\mathsf{I}))-\mathrm{n}_{\mathfrak{u}_{l_3}}(\mathfrak{u}_{l_1}\mathfrak{u}_{l_3}
|(\mathcal{G}_{l_1}^\mathsf{S}\vartriangleright(\mathcal{G}_{l_2}^\mathsf{V}\cup \mathcal{G}_{l_3}^\mathsf{I}))\right|
\\
&+\sum\limits_{\mathfrak{u}_{l_1}\in \mathsf{V}(\mathcal{G}_{l_1})}\sum\limits_{\mathfrak{u'}_{l_1}\in \mathsf{E}(\mathcal{G}_{l_1})}\left|\mathrm{n}_{\mathfrak{u}_{l_1}}(\mathfrak{u}_{l_1}\mathfrak{u'}_{l_1}
|(\mathcal{G}_{l_1}^\mathsf{S}\vartriangleright(\mathcal{G}_{l_2}^\mathsf{V}\cup \mathcal{G}_{l_3}^\mathsf{I}))-\mathrm{n}_{\mathfrak{u'}_{l_1}}(\mathfrak{u}_{l_1}\mathfrak{u'}_{l_1}
|(\mathcal{G}_{l_1}^\mathsf{S}\vartriangleright(\mathcal{G}_{l_2}^\mathsf{V}\cup \mathcal{G}_{l_3}^\mathsf{I}))\right|.
\end{split}
\end{eqnarray}

Since the SVE of three graphs has diameter at most $3$.
Observe that if $\mathfrak{u}_{l_2}\mathfrak{u'}_{l_2}\in \mathsf{E}(\mathcal{G}_{l_2})$ then we have $\mathrm{n}_{\mathfrak{u}_{l_2}}(\mathfrak{e}=\mathfrak{u}_{l_2}\mathfrak{u'}_{l_2}|(\mathcal{G}_{l_1}^
\mathsf{S}\vartriangleright(\mathcal{G}_{l_2}^\mathsf{V}\cup \mathcal{G}_{l_3}^\mathsf{I}))
=\deg_{\mathcal{G}_{l_2}}(\mathfrak{u}_{l_2})-|N_{\mathcal{G}_{l_2}}(\mathfrak{u}_{l_2}\mathfrak{u'}_{l_2})|$ and $\mathrm{n}_{\mathfrak{u'}_{l_2}}(\mathfrak{e}=\mathfrak{u}_{l_2}\mathfrak{u'}_{l_2}|(\mathcal{G}_{l_1}^
\mathsf{S}\vartriangleright(\mathcal{G}_{l_2}^\mathsf{V}\cup \mathcal{G}_{l_3}^\mathsf{I}))=
\deg_{\mathcal{G}_{l_2}}(\mathfrak{u'}_{l_2})-|N_{\mathcal{G}_{l_2}}(\mathfrak{u}_{l_2}\mathfrak{u'}_{l_2})|$.
Analogous if $\mathfrak{u}_{l_3}\mathfrak{u'}_{l_3}\in \mathsf{E}(\mathcal{G}_{l_3})$ then we have $\mathrm{n}_{\mathfrak{u}_{l_3}}(\mathfrak{e}=\mathfrak{u}_{l_3}\mathfrak{u'}_{l_3}|(\mathcal{G}_{l_1}^
\mathsf{S}\vartriangleright(\mathcal{G}_{l_2}^\mathsf{V}\cup \mathcal{G}_{l_3}^\mathsf{I}))=
\deg_{\mathcal{G}_{l_3}}(\mathfrak{u}_{l_3})-|N_{\mathcal{G}_{l_3}}(\mathfrak{u}_{l_3}\mathfrak{u'}_{l_3})|$ and $\mathrm{n}_{\mathfrak{u'}_{l_3}}(\mathfrak{e}=\mathfrak{u}_{l_3}\mathfrak{u'}_{l_3}|(\mathcal{G}_{l_1}^\mathsf{S}
\vartriangleright(\mathcal{G}_{l_2}^\mathsf{V}\cup \mathcal{G}_{l_3}^\mathsf{I}))
=\deg_{\mathcal{G}_{l_3}}(\mathfrak{u'}_{l_3})-|N_{\mathcal{G}_{l_3}}(\mathfrak{u}_{l_3}\mathfrak{u'}_{l_3})|$.
Therefore
\begin{eqnarray}\label{sve1}
\begin{split}
&\sum\limits_{\mathfrak{u}_{l_2}\mathfrak{u'}_{l_2}\in \mathsf{E}(\mathcal{G}_{l_2})}\left|\mathrm{n}_{\mathfrak{u}_{l_2}}(\mathfrak{u}_{l_2}\mathfrak{u'}_{l_2}
|(\mathcal{G}_{l_1}^\mathsf{S}\vartriangleright(\mathcal{G}_{l_2}^\mathsf{V}\cup \mathcal{G}_{l_3}^\mathsf{I}))-\mathrm{n}_{\mathfrak{u'}_{l_2}}(\mathfrak{u}_{l_2}\mathfrak{u'}_{l_2}
|(\mathcal{G}_{l_1}^\mathsf{S}\vartriangleright(\mathcal{G}_{l_2}^\mathsf{V}\cup \mathcal{G}_{l_3}^\mathsf{I}))\right|
\\
=&\sum\limits_{\mathfrak{u}_{l_2}\mathfrak{u'}_{l_2}\in \mathsf{E}(\mathcal{G}_{l_2})}\left|\deg_{\mathcal{G}_{l_2}}(\mathfrak{u}_{l_2})-|N_{\mathcal{G}_{l_2}}
(\mathfrak{u}_{l_2}\mathfrak{u'}_{l_2})|-\deg_{\mathcal{G}_{l_2}}(\mathfrak{u'}_{l_2})+|N_{\mathcal{G}_{l_2}}
(\mathfrak{u}_{l_2}\mathfrak{u'}_{l_2})|\right|
\\
=&irr(\mathcal{G}_{l_2}).
\end{split}
\end{eqnarray}
Similarly
\begin{eqnarray}\label{sve2}
\begin{split}
&\sum\limits_{\mathfrak{u}_{l_3}\mathfrak{u'}_{l_3}\in \mathsf{E}(\mathcal{G}_{l_3})}\left|\mathrm{n}_{\mathfrak{u}_{l_3}}(\mathfrak{u}_{l_3}\mathfrak{u'}_{l_3}|
(\mathcal{G}_{l_1}^\mathsf{S}\vartriangleright(\mathcal{G}_{l_2}^\mathsf{V}\cup \mathcal{G}_{l_3}^\mathsf{I}))-\mathrm{n}_{\mathfrak{u'}_{l_3}}(\mathfrak{u}_{l_3}\mathfrak{u'}_{l_3}
|(\mathcal{G}_{l_1}^\mathsf{S}\vartriangleright(\mathcal{G}_{l_2}^\mathsf{V}\cup \mathcal{G}_{l_3}^\mathsf{I}))\right|
\\
=&\sum\limits_{\mathfrak{u}_{l_3}\mathfrak{u'}_{l_3}\in \mathsf{E}(\mathcal{G}_{l_3})}\left|\deg_{\mathcal{G}_{l_3}}(\mathfrak{u}_{l_3})-|N_{\mathcal{G}_{l_3}}
(\mathfrak{u}_{l_3}\mathfrak{u'}_{l_3})|-\deg_{\mathcal{G}_{l_3}}(\mathfrak{u'}_{l_3})+|N_{\mathcal{G}_{l_3}}
(\mathfrak{u}_{l_3}\mathfrak{u'}_{l_3})|\right|
\\
=&irr(\mathcal{G}_{l_3}).
\end{split}
\end{eqnarray}
Also if $\mathfrak{u}_{l_1}\in \mathsf{V}(\mathcal{G}_{l_1})$ and $\mathfrak{u}_{l_2}\in \mathsf{V}(\mathcal{G}_{l_2})$ then we have $\mathrm{n}_{\mathfrak{u}_{l_1}}(\mathfrak{e}=\mathfrak{u}_{l_1}\mathfrak{u}_{l_2}|\mathcal{G}_{l_1}^\mathsf{S}
\vartriangleright(\mathcal{G}_{l_2}^\mathsf{V}\cup \mathcal{G}_{l_3}^\mathsf{I}))
=s_2+s_3-\deg_{\mathcal{G}_{l_2}}(\mathfrak{u}_{l_2})+\deg_{\mathcal{G}_{l_1}}(\mathfrak{u}_{l_1})$ and $\mathrm{n}_{\mathfrak{u}_{l_2}}(\mathfrak{e}=\mathfrak{u}_{l_1}\mathfrak{u}_{l_2}|\mathcal{G}_{l_1}^\mathsf{S}
\vartriangleright(\mathcal{G}_{l_2}^\mathsf{V}\cup \mathcal{G}_{l_3}^\mathsf{I}))
=s_1-\deg_{\mathcal{G}_{l_1}}(\mathfrak{u}_{l_1})+t_1$.
Therefore
\begin{eqnarray}\label{sve3}
\begin{split}
&\sum\limits_{\mathfrak{u}_{l_1}\in \mathsf{V}(\mathcal{G}_{l_1})}\sum\limits_{\mathfrak{u}_{l_2}\in \mathsf{V}(\mathcal{G}_{l_2})}\left|\mathrm{n}_{\mathfrak{u}_{l_1}}(\mathfrak{u}_{l_1}\mathfrak{u}_{l_2}
|\mathcal{G}_{l_1}^\mathsf{S}\vartriangleright(\mathcal{G}_{l_2}^\mathsf{V}\cup \mathcal{G}_{l_3}^\mathsf{I}))-\mathrm{n}_{\mathfrak{u}_{l_2}}(\mathfrak{u}_{l_1}\mathfrak{u}_{l_2}
|\mathcal{G}_{l_1}^\mathsf{S}\vartriangleright(\mathcal{G}_{l_2}^\mathsf{V}\cup \mathcal{G}_{l_3}^\mathsf{I}))\right|
\\
&=\sum\limits_{\mathfrak{u}_{l_1}\in \mathsf{V}(\mathcal{G}_{l_1})}\sum\limits_{\mathfrak{u}_{l_2}\in \mathsf{V}(\mathcal{G}_{l_2})}\left|s_2+s_3-\deg_{\mathcal{G}_{l_2}}(\mathfrak{u}_{l_2})
+\deg_{\mathcal{G}_{l_1}}(\mathfrak{u}_{l_1})-s_1+\deg_{\mathcal{G}_{l_1}}(\mathfrak{u}_{l_1})-t_1\right|
\\
&\leq\sum\limits_{\mathfrak{u}_{l_1}\in \mathsf{V}(\mathcal{G}_{l_1})}\sum\limits_{\mathfrak{u}_{l_2}\in \mathsf{V}(\mathcal{G}_{l_2})}\left|s_2+s_3-s_1-t_1\right|+2\sum\limits_{\mathfrak{u}_{l_1}\in \mathsf{V}(\mathcal{G}_{l_1})}\sum\limits_{\mathfrak{u}_{l_2}\in \mathsf{V}(\mathcal{G}_{l_2})}\deg_{\mathcal{G}_{l_1}}(\mathfrak{u}_{l_1})\\
&+\sum\limits_{\mathfrak{u}_{l_1}\in \mathsf{V}(\mathcal{G}_{l_1})}\sum\limits_{\mathfrak{u}_{l_2}\in \mathsf{V}(\mathcal{G}_{l_2})}\deg_{\mathcal{G}_{l_2}}(\mathfrak{u}_{l_2})
\\
&=s_1s_2\left|s_2+s_3-s_1-t_1\right|+4t_1s_2+2s_1t_2.
\end{split}
\end{eqnarray}
Also if $\mathfrak{u}_{l_1}\in \mathsf{E}(\mathcal{G}_{l_1})$ and $\mathfrak{u}_{l_3}\in \mathsf{V}(\mathcal{G}_{l_3})$ then we have $\mathrm{n}_{\mathfrak{u}_{l_1}}(\mathfrak{e}=\mathfrak{u}_{l_1}\mathfrak{u}_{l_3}|\mathcal{G}_{l_1}^
\mathsf{S}\vartriangleright(\mathcal{G}_{l_2}^\mathsf{V}\cup \mathcal{G}_{l_3}^\mathsf{I}))
=s_3-\deg_{\mathcal{G}_{l_3}}(\mathfrak{u}_{l_3})+2+s_2$ and $\mathrm{n}_{\mathfrak{u}_{l_3}}(\mathfrak{e}=\mathfrak{u}_{l_1}\mathfrak{u}_{l_3}|\mathcal{G}_{l_1}^\mathsf{S}
\vartriangleright(\mathcal{G}_{l_2}^\mathsf{V}\cup \mathcal{G}_{l_3}^\mathsf{I}))=s_1+t_1-2$.
Therefore
\begin{eqnarray}\label{sve4}
\begin{split}
&\sum\limits_{\mathfrak{u}_{l_1}\in \mathsf{E}(\mathcal{G}_{l_1})}\sum\limits_{\mathfrak{u}_{l_3}\in \mathsf{V}(\mathcal{G}_{l_3})}\left|\mathrm{n}_{\mathfrak{u}_{l_1}}(\mathfrak{u}_{l_1}\mathfrak{u}_{l_3}
|\mathcal{G}_{l_1}^\mathsf{S}\vartriangleright(\mathcal{G}_{l_2}^\mathsf{V}\cup \mathcal{G}_{l_3}^\mathsf{I}))-\mathrm{n}_{\mathfrak{u}_{l_3}}(\mathfrak{u}_{l_1}\mathfrak{u}_{l_3}
|\mathcal{G}_{l_1}^\mathsf{S}\vartriangleright(\mathcal{G}_{l_2}^\mathsf{V}\cup \mathcal{G}_{l_3}^\mathsf{I}))\right|
\\
&=\sum\limits_{\mathfrak{u}_{l_1}\in \mathsf{E}(\mathcal{G}_{l_1})}\sum\limits_{\mathfrak{u}_{l_3}\in \mathsf{V}(\mathcal{G}_{l_3})}\left|s_2+s_3+2-\deg_{\mathcal{G}_{l_3}}(\mathfrak{u}_{l_3})-s_1-t_1+2\right|
\\
&\leq\sum\limits_{\mathfrak{u}_{l_1}\in \mathsf{E}(\mathcal{G}_{l_1})}\sum\limits_{\mathfrak{u}_{l_3}\in \mathsf{V}(\mathcal{G}_{l_3})}\left|s_2+s_3-s_1-t_1+4\right|+\sum\limits_{\mathfrak{u}_{l_1}\in \mathsf{E}(\mathcal{G}_{l_1})}\sum\limits_{\mathfrak{u}_{l_3}\in \mathsf{V}(\mathcal{G}_{l_3})}\deg_{\mathcal{G}_{l_3}}(\mathfrak{u}_{l_3})
\\
&=t_1s_3\left|s_2+s_3-s_1-t_1+4\right|+2t_1t_3.
\end{split}
\end{eqnarray}
Finally if $\mathfrak{u}_{l_1}\in \mathsf{V}(\mathcal{G}_{l_1})$ and $\mathfrak{u'}_{l_1}\in \mathsf{E}(\mathcal{G}_{l_1})$ then we have $\mathrm{n}_{\mathfrak{u}_{l_1}}(\mathfrak{e}=\mathfrak{u}_{l_1}\mathfrak{u'}_{l_1}|\mathcal{G}_{l_1}^
\mathsf{S}\vartriangleright(\mathcal{G}_{l_2}^\mathsf{V}\cup \mathcal{G}_{l_3}^\mathsf{I}))
=s_2+s_1-2+\deg_{\mathcal{G}_{l_1}}(\mathfrak{u'}_{l_1})$ and $\mathrm{n}_{\mathfrak{u'}_{l_1}}(\mathfrak{e}=\mathfrak{u}_{l_1}\mathfrak{u'}_{l_1}|\mathcal{G}_{l_1}^
\mathsf{S}\vartriangleright(\mathcal{G}_{l_2}^\mathsf{V}\cup \mathcal{G}_{l_3}^\mathsf{I}))=t_1+s_3+2-\deg_{\mathcal{G}_{l_1}}(\mathfrak{u}_{l_1})$.
Therefore
\begin{eqnarray}\label{sve5}
\begin{split}
&\sum\limits_{\mathfrak{u}_{l_1}\in \mathsf{V}(\mathcal{G}_{l_1})}\sum\limits_{\mathfrak{u'}_{l_1}\in \mathsf{E}(\mathcal{G}_{l_1})}\left|\mathrm{n}_{\mathfrak{u}_{l_1}}(\mathfrak{u}_{l_1}\mathfrak{u'}_{l_1}
|\mathcal{G}_{l_1}^\mathsf{S}\vartriangleright(\mathcal{G}_{l_2}^\mathsf{V}\cup \mathcal{G}_{l_3}^\mathsf{I}))-\mathrm{n}_{\mathfrak{u'}_{l_1}}(\mathfrak{u}_{l_1}\mathfrak{u'}_{l_1}
|\mathcal{G}_{l_1}^\mathsf{S}\vartriangleright(\mathcal{G}_{l_2}^\mathsf{V}\cup \mathcal{G}_{l_3}^\mathsf{I}))\right|
\\
&=\sum\limits_{\mathfrak{u}_{l_1}\in \mathsf{V}(\mathcal{G}_{l_1})}\sum\limits_{\mathfrak{u'}_{l_1}\in \mathsf{E}(\mathcal{G}_{l_1})}\left|s_2+s_1-2+\deg_{\mathcal{G}_{l_1}}(\mathfrak{u'}_{l_1})-t_1-s_3-2
+\deg_{\mathcal{G}_{l_1}}(\mathfrak{u}_{l_1})\right|
\\
&\leq\sum\limits_{\mathfrak{u}_{l_1}\in \mathsf{V}(\mathcal{G}_{l_1})}\sum\limits_{\mathfrak{u'}_{l_1}\in \mathsf{E}(\mathcal{G}_{l_1})}\left|s_2+s_1-s_3-t_1-4\right|+2\sum\limits_{\mathfrak{u}_{l_1}\in \mathsf{V}(\mathcal{G}_{l_1})}\sum\limits_{\mathfrak{u'}_{l_1}\in \mathsf{E}(\mathcal{G}_{l_1})}\deg_{\mathcal{G}_{l_1}}(\mathfrak{u}_{l_1})
\\
&=t_1s_1\left|s_2+s_1-s_3-t_1-4\right|+4t^2_1.
\end{split}
\end{eqnarray}
Therefore by using above results \eqref{sve1}-\eqref{sve5} in \eqref{sve}, we obtain the required expression for $Mo(\mathcal{G}_{l_1}^\mathsf{S}\vartriangleright(\mathcal{G}_{l_2}^\mathsf{V}\cup \mathcal{G}_{l_3}^\mathsf{I}))$.
\end{proof}

\section{Conclusion}

The present-day trend of the numerical coding of chemical structures with topological descriptors has proven quite successful in Bioinformatics and Chemistry.
This scheme yields the retrieval, mining, rapid collection, annotation, and comparison of chemical structures within large databases. Subsequently, topological descriptors can be applied to study for QSAR and QSPR, which are models, that correlate chemical structure with physical properties, biological activity or chemical reactivity. In this article, we have given the results for the Mostar index of corona product, Cartesian product, join, lexicographic product, Indu-Bala product and subdivision vertex-edge join of graphs and apply these outcomes to find the Mostar index of various classes of chemical graphs and nanostructures.

\section*{Conflict of Interests}
The authors hereby declare that there is no conflict of interests regarding the publication of this paper.

\section*{Acknowledgment}

We thank the reviewers for their constructive comments in improving the quality of this paper.

\end{document}